%% file: main.tex
\documentclass[table,11pt]{article}
\usepackage{setspace}
\usepackage[margin=1in]{geometry}
\usepackage{comment}
\usepackage{hyperref}
\hypersetup{
  colorlinks=true,
  linkcolor=black,
  citecolor=black,
  urlcolor=blue,
}
\usepackage{amsmath}    
\usepackage{amssymb}    
\usepackage{amsthm}     
\usepackage{mathtools}
\usepackage{cleveref}
\usepackage{xfrac}
\usepackage{algorithm}
\usepackage{algpseudocode}
\usepackage{float}

\DeclareMathOperator{\vertices}{vert}

\newtheorem{theorem}{Theorem}

\newtheorem{corollary}[theorem]{Corollary}
\newtheorem{proposition}[theorem]{Proposition}

\usepackage{xcolor}
\usepackage[round]{natbib}
\usepackage{titlecaps}
\usepackage{titlesec}

\usepackage{import}

\titleformat{\section}
  {\normalfont\fontsize{14}{15}\selectfont\bfseries}
  {\thesection}{0.5em}{}

\newcommand{\R}{\mathbb{R}}

\usepackage{nicefrac}
\usepackage{tikz}
\usepackage{pgfplots}
\pgfplotsset{compat=1.18}
\pgfplotsset{
  table/search path={../manuscript/}
}

\usetikzlibrary{
  pgfplots.groupplots,
  patterns
}
\usepackage{booktabs}

\colorlet{blueI}{gray!5}   
\colorlet{blueII}{gray!20}  
\colorlet{blueIII}{gray!40} 
\colorlet{blueIV}{gray!60}  

\newcommand{\cI}[1]{\cellcolor{blueI}{#1}}     
\newcommand{\cII}[1]{\cellcolor{blueII}{#1}}   
\newcommand{\cIII}[1]{\cellcolor{blueIII}{#1}} 
\usepackage{makecell}
\usepackage{multirow}
\usepackage{subcaption}
\definecolor{sorGreen}{HTML}{417505}
\definecolor{sorRed}{rgb}{0.784,0.157,0.118}
\definecolor{ORBlue}{HTML}{2F468C}

\usepackage{pgfplotstable}
\usetikzlibrary{pgfplots.groupplots}
\usetikzlibrary{patterns}
\usepackage{graphicx}
\makeatletter
\let\oldproof\proof

\renewcommand{\proof}{%
  \@ifnextchar[{\proof@opt}{\proof@noopt}%
}

\def\proof@opt[#1]{%
  \oldproof[#1]%
  \@ifnextchar\bgroup{\@gobble}{}%
}

\def\proof@noopt{%
  \oldproof
  \@ifnextchar\bgroup{\@gobble}{}%
}
\makeatother

\begin{document}

\def\spacingset#1{\renewcommand{\baselinestretch}%
{#1}\small\normalsize} \spacingset{1}


\begin{center}

\vspace*{-0.5em}

{\Large\bfseries
Network Flexibility Design Under Endogenous Uncertainty
\par}

\vspace{0.7em}

{\large
Soroush Fatemi-Anaraki\textsuperscript{a},
Martin Grunow\textsuperscript{a,b},
Giovanni Pantuso\textsuperscript{c},
Stefan Weltge\textsuperscript{d}
\par}

\vspace{0.65em}

\begin{minipage}{0.88\textwidth}
\small
\setlength{\parindent}{0pt}
\setlength{\parskip}{0.12em}

\textsuperscript{a}\,
TUM School of Management, Technical University of Munich, Germany

\textsuperscript{b}\,
Munich Data Science Institute, Technical University of Munich, Germany

\textsuperscript{c}\,
Department of Mathematical Sciences, University of Copenhagen, Denmark

\textsuperscript{d}\,
Department of Mathematics, Technical University of Munich, Germany

\vspace{0.4em}

\centering
\footnotesize

\href{mailto:soroush.fatemianaraki@tum.de}
{\texttt{soroush.fatemianaraki@tum.de}}
\quad
\href{mailto:martin.grunow@tum.de}
{\texttt{martin.grunow@tum.de}}
\quad
\href{mailto:gp@math.ku.dk}
{\texttt{gp@math.ku.dk}}
\quad
\href{mailto:weltge@tum.de}
{\texttt{weltge@tum.de}}

\end{minipage}

\end{center}

\vspace{0.7em}

\begin{abstract}
\noindent Manufacturing firms face uncertainty in both supply and demand.
They implement process flexibility, enabling plants to produce multiple products, to hedge against supply--demand mismatches.
Assigning more products to a plant results in efficiency losses at the plant, whereas limited flexibility may improve efficiency because of the specialization effect.
Furthermore, the flexibility design determines the sourcing locations of products, thereby affecting demand due to the country-of-origin effect and exposure to tariffs.
We incorporate these mechanisms into the supply network flexibility design problem by modeling supply and demand uncertainties as endogenous, that is, as functions of the first-stage design decisions.
We propose a two-stage stochastic program and solve it via a decomposition scheme that relies on distribution-specific optimality cuts.
To avoid exhaustive enumeration of all distributions, we derive distribution-free formulations that are valid for all distributions and for subsets of supply and demand distributions.
Through extensive computational experiments, we demonstrate the superior performance of our solution approach and distribution-free formulations relative to exhaustive distribution enumeration and state-of-the-art inequalities, respectively.
Moreover, we explain how the problem structure affects the effectiveness of the proposed formulations.
From a managerial perspective, we show the trade-offs between endogenous effects and how they affect the optimal flexibility design.
\end{abstract}

\begin{center}
\begin{minipage}{0.90\textwidth}
\end{minipage}
\end{center}
\vspace{0.6em}

\onehalfspacing

\input{introduction_alt}

\input{literature}
\input{model}
\input{strengthening}
\input{solution}
\input{experiments}
\input{managerial}
\input{conclusion}

\section{Code and Data Disclosure}
The code and data to support the numerical experiments in this paper can be found at \url{https://github.com/soroush-ft/network_flexibility_design}.

\section{Acknowledgements}
This work was funded by the Deutsche Forschungsgemeinschaft (DFG, German Research Foundation) under Grant 277991500/GRK2201 and Novo Nordisk Foundation under grant NNF24OC0089770.

\singlespacing
\bibliography{../manuscript/references}

\appendix
\input{appendix_tables}

\end{document}

%% file: introduction_alt.tex
\section{Introduction}\label{introduction}

Manufacturing firms operate under uncertainty in both supply and demand \citep{simangunsong2012supply,rajaram2002product, bengtsson2002valuation}.
To hedge against the supply-demand mismatch, companies invest in process flexibility, enabling plants to manufacture multiple products \citep{sethi1990flexibility}.
At the supply chain design level, they must determine the \emph{flexibility design}, i.e., the assignment of products to plants, that maximizes expected profits.
The existing literature on the flexibility design problem commonly assumes that supply and demand uncertainties are exogenous, i.e., their distributions are fixed and independent of the flexibility design (see, e.g., \citeauthor{feng2017process}, \citeyear{feng2017process}; \citeauthor{wang2021review}, \citeyear{wang2021review}).
We argue that this assumption is restrictive because flexibility design inevitably affects the uncertainties of both supply and demand. 

From the supply perspective, adding more product variety to a plant results in efficiency losses.
It may reduce available capacity due to changeovers between different products \citep{mak2009stochastic}.
Moreover, products often contain several features offered in multiple options.
Therefore, assigning more products to a plant raises feature-option combinations, which in turn increase the changeover complexity \citep{sun2018car}.
In fact, \cite{choudhary2018flexibility} find that the changeover loss due to the number of products assigned to a plant has a significant negative impact on manufacturing productivity.
Empirical studies on the impact of product variety on the performance of automotive assembly plants further show that parts complexity negatively affects productivity \citep{macduffie1996product}.
In addition, higher daily fluctuations in option content have a negative effect on total labor hours per product and assembly line downtime \citep{fisher1999impact}.
Conversely, limiting the number of products assigned to a plant allows for the use of specialized equipment and workforce, resulting in higher productivity due to the specialization effect.
Building on these empirical observations, we argue that supply uncertainty depends on the network flexibility design.

On the demand side, the choice of flexibility design in a global supply chain affects the sourcing locations of products, thereby impacting their price, perceived quality, and customers' buying intentions.
When importing products manufactured at offshore plants, companies may incur additional costs in the form of tariffs and duties based on the country of origin.
These costs are typically passed on to consumers in the form of inflated prices, which in turn reduces market demand.
This issue became salient in 2018, when the United~States levied tariffs of 10--50\% on \$283 billion worth of imports that led to an increase in product prices of 10--30\% \citep{amiti2019impact}.
\cite{Hinz2026Tariffs} show that in 2025, $96\%$ of tariff costs were passed on to the buyers in the United~States, resulting in a sharp drop in trade volumes.
In the wake of high tariffs and uncertain trade policies, companies ought to reconsider their supply chain strategies \citep{dong2020impact}.
In addition, the image of the country from which a product is sourced affects consumers' evaluation of its quality \citep{verlegh2005country} and their purchase intentions \citep{peterson1995meta}.
Even when consumers are unaware of a product's attributes, they evaluate it positively based on its country of origin \citep{han1989country}.
It is further shown that consumers may have a bias against foreign products in favor of domestically produced alternatives, which is reflected in their perceptions and buying intentions \citep{balabanis2004domestic}.
Consequently, sourcing locations induced by the flexibility design affect market demand.

Given the impact of the flexibility design on both supply and demand, we extend the literature by modeling and solving the problem under endogenous uncertainty.
We formulate it as a two-stage stochastic program such that the second-stage uncertainty depends on the first-stage design decisions.
More formally, let $I$ be the set of plants and $J$ be the set of products.
A flexibility design is a set $E \subseteq I \times J$, where $(i,j) \in E$ indicates that plant $i$ can produce product $j$.
For each pair of plant and product $(i,j)$, its investment cost is $s_{i,j}$, the processing time of a unit of product $j$ at plant $i$ is $r_{i,j}$, and the profit from producing one unit of product $j$ at plant $i$ is $q_{i,j}$.
The production capacity of plant $i$ is a random variable $\Xi_{i}^{c}$, and the market demand for product $j$ is a random variable $\Xi_{j}^{d}$.
We denote by $\Xi$ the joint random vector of supply and demand, and by $\xi$ its realization.
For each realization $\xi$ and a given design $E$, we denote by $f(E, \xi)$ the maximum profit attainable subject to the available supply and demand in the network design.
The standard network design problem seeks to find a design $E^{*}$ that maximizes 
$\mathbb{E} \left[f(E, \Xi)\right] - \sum_{(i,j) \in E} s_{i,j}$, where the expectation is taken with respect to a \emph{fixed} distribution of $\Xi$.
We extend the standard problem to the case where uncertainty is \emph{endogenous}, meaning that the distribution of $\Xi$ depends on the flexibility design $E$.

Incorporating endogenous uncertainty is challenging for two reasons.
Firstly, it is non-trivial to abstract design-level properties of $E$ that capture the relationship between $E$ and the resulting distribution of $\Xi$ while keeping the optimization model computationally tractable.
To this end, we build on empirical work by \cite{choudhary2018flexibility} and use the number of products assigned to each plant (node degree) to determine its supply distribution.
Regarding demand uncertainty, we account for the location of plants supplying a product to determine its demand distribution.
Individual plants have identical impacts on a product's demand if they have similar tariff rates or country-of-origin effects, e.g., if they are located in the same country.
Therefore, we may group such plants into zones, and model demand uncertainty as a function of the sourcing zones of products.
This is without loss of generality, as we may have as many zones as plants.

Secondly, solving the flexibility design problem under endogenous uncertainty is substantially more challenging than the standard case.
The classical L-shaped method \citep{van1969shaped} can be directly applied to solve standard two-stage stochastic programs \citep{birge2011introduction}.
However, when  uncertainty is endogenous, optimality cuts generated at any iteration of the L-shaped method correspond to the specific distribution of supply and demand  induced by the first-stage design.
Hence, the optimality cuts impose a valid upper bound only for the first-stage designs that induce the same distribution.
This is in contrast with the standard two-stage stochastic program where optimality cuts impose valid upper bounds on all first-stage decisions. 
\cite{pantuso2025shaped} extend the L-shaped method by making optimality cuts dependent on the distribution.
Adopting the same paradigm, we implement the L-shaped method and derive \emph{distribution-specific} optimality cuts based on node degrees of the plants and zone-assignments of the products in the first-stage design. 

If the decomposition algorithm only relies on distribution-specific cuts, it must visit all distributions at least once before it converges.
This is highly inefficient, as the number of distributions grows exponentially in the number of products and plants.
To mitigate this limitation, \cite{pantuso2025solutionstochasticfacilitylocation} propose a single inequality that imposes a valid upper bound on  all possible distributions, hence being valid for all first-stage designs.
Furthermore, \cite{pantuso2025shaped} compute an inequality in every iteration of the decomposition approach using the incumbent first-stage solution and a dominant scenario to impose a valid upper bound on all first-stage solutions.
These findings represent the state of the art and serve as natural benchmarks for our decomposition algorithm.

We propose an alternative approach to improve the efficiency of the decomposition algorithm and derive three new classes of \emph{distribution-free} formulations that strengthen the relaxed master problem using a small number of inequalities. 
We embed the second-stage problem in the relaxed master problem, and impose upper bounds on the expected second-stage profit that are valid for all distributions, subsets of supply distributions, and subsets of demand distributions.
The variant of our distribution-free formulations that is valid for all distributions implies the state-of-the-art cuts proposed by \cite{pantuso2025shaped}.
Moreover, they imply the cut proposed by \cite{pantuso2025solutionstochasticfacilitylocation} when the supply and demand distributions are independent from each other.

Our contributions are as follows. 
We extend the flexibility design literature by relaxing the assumption that the second-stage uncertainty is exogenous and model the problem as a two-stage stochastic program with endogenous uncertainty.
Our modeling approach is generic and captures the impact of two distinct design properties, namely, the degrees and zone assignments of nodes, on the second-stage uncertainty.
To solve the problem, we implement the L-shaped method and derive distribution-specific optimality cuts.
This implementation serves as our baseline approach.
To avoid exhaustive distribution enumeration, we propose three new classes of distribution-free formulations that strengthen the relaxed master problem, thereby significantly improving the efficiency of the baseline approach.

We perform extensive numerical experiments to evaluate the performance of our baseline approach and distribution-free formulations.
First, we compare the baseline approach against the exhaustive distribution-enumeration approach to highlight its efficiency.
While both approaches must visit all distributions, the exhaustive distribution-enumeration approach is substantially slower;
it requires, on average, $1.77$ times as many iterations and $2.44$ times as much solution time to converge. 
Second, we augment the baseline approach with our distribution-free formulations and compare the performance against the state-of-the-art inequalities.
We investigate three uncertainty regimes, namely, (1) endogenous supply and demand, (2) exogenous supply and endogenous demand, and (3) endogenous supply and exogenous demand.
Experiments confirm that our distribution-free formulations improve the efficiency of the baseline approach both in terms of solution time and the optimality gap, under all uncertainty regimes.
Furthermore, they outperform the state-of-the-art inequalities under uncertainty Regimes~(1) and~(2) while showing comparable performance under Regime~(3).

Third, through a sensitivity analysis, we demonstrate how the right-hand side of the constraints  used in the distribution-free formulations, as determined by the endogenous uncertainty mechanism, affect their relative performance.
Based on this analysis, if the problem structure does not clearly favor a particular class of formulations, it is most reliable to use the combination of all distribution-free formulations.
Fourth, we show that our approach computes high-quality lower bounds in the early stages of the solution process for two representative instances.

For an illustrative example, we demonstrate how endogenous supply and demand effects jointly shape the optimal network design. 
Considering low and high levels of supply variability, we show that endogenous effects can lead to either more or less flexible designs than the exogenous benchmark (the optimal design excluding endogenous effects).
In our example, the optimal design under endogenous uncertainty achieves an objective value improvement of up to $3.2\%$ and $2.82\%$ relative to the exogenous benchmark under low and high supply-variability settings, respectively.
Moreover, our analysis outlines clear thresholds at which the optimal design changes as the endogenous effects vary, thereby providing insight into the robustness of the optimal design to perturbations in the parameters governing these effects.

The rest of the paper is organized as follows.
In Section~\ref{sec:literature}, we review the literature on the flexibility design problem, modeling tariffs in supply network planning, and modeling stochastic problems under endogenous uncertainty.
In Section~\ref{sec:model}, our two-stage stochastic programming formulation is provided.
Section~\ref{sec:strengthening} introduces our distribution-free formulations.
The L-shaped method with distribution-specific optimality cuts is presented in Section~\ref{sec:solution_approach}.  
Sections~\ref{sec:experiments} and~\ref{sec:managerial} present computational and managerial insights, respectively.
We conclude the paper in Section~\ref{sec:conclusion}.

%% file: literature.tex
\section{Literature Review}\label{sec:literature}

\subsection{Flexibility Design Problem}\label{sec:intro_literature}

The flexibility design literature assumes the uncertainty in the second stage is exogenous.
In this setting, increasing flexibility improves the ability of the supply network to match supply and demand.
Then, a fully flexible design ($E = I \times J$) maximizes future sales.
However, it is costly and difficult to implement.
Therefore, a stream of literature focuses on sparse designs that perform well under uncertainty.
\cite{jordan1995principles} consider a production system under demand uncertainty with identical plants and products, and propose the well-known long-chain design, wherein each plant (product) is assigned to two products (plants), such that the whole design is connected.
The long-chain design has received substantial attention in the literature (see, e.g., \citeauthor{chou2010design}, \citeyear{chou2010design}; \citeauthor{simchi2012understanding}, \citeyear{simchi2012understanding}; \citeauthor{wang2015process}, \citeyear{wang2015process}; \citeauthor{desir2016sparse}, \citeyear{desir2016sparse}).

Several studies incorporate heterogeneity across plants, products, and unit profits, and include the investment costs.
\cite{mak2009stochastic} propose a two-stage stochastic programming formulation and solve it using a Lagrangian-based heuristic.
\cite{feng2017process} generalize the problem setting to the case where processing times differ for each pair of plants and products.
They model the problem as a two-stage stochastic program and propose an accelerated Bender's decomposition approach and a fix-and-optimize heuristic to solve it.
Overall, the modeling literature is sparser than the analytical stream (see \cite{wang2021review} for a review).
Our work falls into the modeling stream and proposes a two-stage stochastic programming formulation in which second-stage supply and demand uncertainties are endogenous and depend on the first-stage design.

\subsection{Modeling Tariff Effects}\label{sec:literature_tariffs}

Modeling trade barriers in flexibility design is a pressing issue given recent geopolitical events.
This is a challenging task because allocation quantities of production volumes to individual plants are only determined once the uncertainties of supply and demand realize, long after the investment decisions were made.
On the other hand, knowing these allocation quantities is crucial to understanding the true costs of offshoring and the resulting impact on product prices.
This problem simplifies under the assumption of fixed prices, which allows exogenous modeling of demand and reflects tariff effects by decreasing the company's profit margin when sourcing products from offshore plants (see, e.g., \citeauthor{lu2009multimarket}, \citeyear{lu2009multimarket}).

Empirical evidence suggests companies pass tariffs to consumers by increasing prices \citep{amiti2019impact,Hinz2026Tariffs}.
Therefore, assuming fixed prices in the presence of tariffs may be limiting.
\cite{dong2010global} consider a capacity investment and distribution problem for a firm selling a single product in two markets, facing trade costs for transshipment and uncertainties in market size and exchange rate.
Once the firm sets production levels, responsive pricing (as a decreasing function of the supply quantity and market size) is used to clear the market.
This decision sequence implies that market size uncertainty is independent of the company's price choices.
Accounting for market competition, several studies consider stylized examples in which Cournot competition determines market-clearing prices (see, e.g., \citeauthor{dong2020impact}, \citeyear{dong2020impact}; \citeauthor{wei2025optimizing}, \citeyear{wei2025optimizing}).
These ex-post pricing approaches are not suitable for our setting because they require explicit modeling of downstream competition and pricing mechanisms at the supply network investment stage, which relies on strong structural assumptions and substantially increases modeling and computational complexity.
Moreover, it is not clear how the stylized settings in the literature can be transferred to a generic supply network design problem.

The existing literature primarily captures tariffs via changes in unit costs and employs different pricing mechanisms to reflect a firm's decision to pass them on to consumers.
In contrast, we avoid explicit pricing decisions at the supply network design level, and instead capture tariff exposure through demand shifts.
The demand distribution may shift unfavorably (e.g., by decreasing the mean value and increasing the level of uncertainty) when products are sourced from offshore plants subject to tariffs.
The change in demand distribution reflects the increase in market prices due to tariffs, while allowing us to maintain the profit margin of a product at a given plant ($q_{i,j}$) constant regardless of the sourcing location. 
This modeling approach provides a supply-network-level representation of tariff risks without requiring explicit pricing or modeling market competition.

\subsection{Conceptual Background: Endogenous Uncertainty}\label{sec:literature_endogenous}

Decisions may influence uncertainty in two ways \citep{goel2006class}: by altering the underlying probability distribution or by determining when uncertainty resolves (see \citeauthor{apap2017models}, \citeyear{apap2017models}).
Pertaining to our problem setting, we focus on the former.

Decisions may affect the probability distributions in different ways. 
When the set of scenarios is fixed (i.e., the support of a discrete distribution), decisions may alter their probabilities either in the objective function (see, e.g., \citeauthor{peeta2010pre}, \citeyear{peeta2010pre}; \citeauthor{viswanath2004investing}, \citeyear{viswanath2004investing}; \citeauthor{galiullina2024demand}, \citeyear{galiullina2024demand}) or via modeling probabilities as decision variables in the model (see, e.g., \citeauthor{laumanns2014distribution}, \citeyear{laumanns2014distribution}). 
These approaches are suitable when the uncertainty is discrete and first-stage decisions make some scenarios more probable.
\cite{hellemo2018decision} further model scenario probabilities as affine functions of first-stage decisions.
\cite{dupacova_optimization_2006}, \cite{JonW}, \cite{hewitt2024production}, \cite{pantuso2025solutionstochasticfacilitylocation}, and \cite{pantuso2025shaped} illustrate a class of stochastic programs where the distribution depends on the decisions in a piecewise-constant manner. That is, there exist finitely many potential distributions that may materialize, and different subregions of the feasible space induce different distributions. 
This type of problem is also known as \emph{selection of distribution}.
Overall, the resulting mathematical programming formulation depends strongly on the relationship between decisions and uncertainty, yielding a heterogeneous class of problems.

In some cases, stochastic programs with endogenous uncertainty can be transformed into traditional stochastic programming problems through a change of measure (see, e.g., \citeauthor{VarW88}, \citeyear{VarW88}; \citeauthor{pfl96}, \citeyear{pfl96}; \citeauthor{rubinstein1993discrete}, \citeyear{rubinstein1993discrete}; \citeauthor{dupacova_optimization_2006}, \citeyear{dupacova_optimization_2006}).
However, desirable properties of the objective function are typically lost in this transition and scalability is severely compromised (see, e.g., the discussion in \cite{VarW88}, \cite{pfl96}, \cite{dupacova_optimization_2006}, and examples in \cite{bazotte_solving_2024}).

Existing solution methods therefore apply only to specific classes of stochastic programs with endogenous uncertainty.
We build upon the extension of the L-shaped method proposed by \cite{pantuso2025shaped} for selection of distribution problems.
The method extends the classical L-shaped method of \cite{van1969shaped} by deriving \textit{distribution-specific} feasibility and optimality cuts to be employed in the decomposition scheme and proposes valid inequalities to speed up convergence.
We extend this method to solve the proposed network flexibility design problem.

%% file: model.tex
\section{Mathematical Model}\label{sec:model}

We proceed now to formally define the flexibility design problem under endogenous uncertainty as a two-stage stochastic program.
Let $I$ be the set of plants and $J$ be the set of products.
Furthermore, let $s_{i,j}$ denote the fixed investment cost for enabling plant $i$ to produce product $j$; $q_{i,j}$ and $r_{i,j}$ denote the profit and processing time per unit of product $j$ produced at plant $i$, respectively.
We define $y \in \{0,1\}^{I \times J}$ as the first-stage design variable, where $y_{i,j} = 1$ if plant $i$ produces product $j$ and $0$ otherwise.
In other words, $y_{i,j}$ indicates whether edge $(i,j)$ is included in the flexibility design $E \subseteq I \times J$.
Recall that each flexibility design $E$ determines a distribution $\mathbb{P}_E$ of supply and demand.
Consequently, to each vector $y$ we associate the corresponding distribution and denote it by $\mathbb{P}_y$.
In what follows, we use $\Xi$ to denote the random vector of supply and demand and $\xi$ to denote a specific realization of it.
More specifically, the total production capacity of plant $i$ is the random variable $\Xi_{i}^{c}$, and the total market demand for product $j$ is the random variable $\Xi_{j}^{d}$.
The generic two-stage stochastic program for this problem is
\begin{align}\label{eq:generic_model}
    \max_{y \in \{0,1\}^{I\times J}} \left\{\mathbb{E}_{\Xi \sim \mathbb{P}_y} \Big[ f(y, \Xi) \Big] - \sum_{i\in I} \sum_{j \in J} s_{i,j} \; y_{i,j} \right\},
\end{align}
maximizing the expected second-stage profit minus the investment costs incurred by the first-stage design.
Given a design $y$ and a realization $\xi$, the second-stage problem is defined as
\begin{equation}\label{eq:max_flow}
    \begin{aligned}
        f(y, \xi) \coloneqq  \max_{x}\ &\sum_{i\in I} \sum_{j \in J} q_{i,j} \; x_{i,j}\\
                            \text{s.t.}\  &\sum_{j \in J} r_{i,j} \; x_{i,j} \le \xi_{i}^{c} &&\forall \ i \in I, \\
                        & \sum_{i \in I} x_{i,j} \le \xi_{j}^{d} &&\forall \ j \in J, \\
                        & \; 0 \le x_{i,j} \le M_{i,j}^\xi \; y_{i,j} &&\forall \ i \in I, \, j \in J
    \end{aligned}
\end{equation}
for $M_{i,j}^\xi = \min \{\nicefrac{\xi_i^c}{r_{i,j}}, \xi_j^d\}$.
Variable $x_{i, j}$ determines the production quantity of product $j$ at plant $i$, which is only possible if $y_{i,j} = 1$ (the domain constraints).
The objective function maximizes the profit, the first set of constraints ensures that the plant capacities are respected, and the second set of constraints bounds total production volumes of each product by the market demand.
The set of all possible distributions in the second-stage problem is denoted by 
$\mathcal{P} \coloneqq \{ \mathbb{P}_y \mid y \in \{0,1\}^{I \times J} \}$.

We first reformulate Problem~\eqref{eq:generic_model} into a more explicit form that will serve as the basis for the mixed-integer linear programming reformulations presented next.
As common in the stochastic programming literature, we approximate the expected second-stage profit through sample average approximation.
To this end, for each $p \in \mathcal{P}$, let $S_p$ be a set of scenarios independently sampled from distribution $p$.
We approximate $\mathbb{E}_{\Xi \sim \mathbb{P}_y} \Big[ f(y, \Xi) \Big]$ by the sample average $F(y,p) \coloneq \nicefrac{1}{|S_p|} \sum_{\xi \in S_p} f(y, \xi)$ where $p=\mathbb{P}_y$.
Next, we replace Problem~\eqref{eq:generic_model} by 
$\max_{y \in \{0,1\}^{I\times J}} \left\{F(y,p) - \sum_{i\in I} \sum_{j \in J} s_{i,j} \; y_{i,j}: p = \mathbb{P}_y \right\}$.
To remove the dependence between $y$ and $p$ from the objective function, we introduce an auxiliary variable $\mu$ that represents the expected second-stage profit:  
\begin{align}
    \max_{y \in \{0,1\}^{I \times J}, \mu} \ &\mu - \sum_{i\in I} \sum_{j \in J} s_{i,j} \; y_{i,j} \nonumber \\ 
    \text{s.t.} \ &\mu \le F(y,p) \ \vee \ \mathbb{P}_y \ne p \qquad \forall \ p \in \mathcal{P}, \label{eq:p_cuts}
\end{align}

Letting $U$ be a sufficiently large number, we may obtain an explicit formulation from the disjunctive constraints in~\eqref{eq:p_cuts} using the inequalities  
\begin{equation}\label{eq:no_good_cut}
    \mu \le \Big[\sum_{(i,j)\in I \times J: \hat{y}_{i,j}=1} \big(1 - y_{i,j} \big) + \sum_{(i,j)\in I \times J:\hat{y}_{i,j}=0} y_{i,j}\Big]\cdot U + F(y, \mathbb{P}_{\hat{y}})
    \quad \forall \hat{y} \in \{0,1\}^{I \times J}.
\end{equation}
This formulation consists of an inequality for every design choice $\hat{y}$, resulting in $2^{|I| \cdot |J|}$ many inequalities, which is highly impractical.
However, we next show how to exploit specific structures of endogenous uncertainty to aggregate many of these inequalities into a more compact form.

\subsection{Reformulating Endogenous Supply Uncertainty}\label{sec:endo_supply_formulation}

For ease of exposition, we momentarily treat demand as exogenous.
In the case of endogenous supply uncertainty, we motivated that the supply distribution in the second-stage problem only depends on the number of products assigned to each plant in the flexibility design.
Hence, the distribution resulting from a design $y$ is determined by the degree of each plant node $i \in I$ in $y$.
For a degree vector $d \in \mathbb{Z}_{\ge 0}^I$, let $\mathbb{P}_d$ denote the supply distribution induced when each plant $i$ is assigned to exactly $d_i$ products.
For each $i \in I$ and $k \in \{0, \dots, |J|\}$, let us introduce a binary variable $w_{i,k} \in \{0,1\}$ indicating whether plant $i$ is assigned to exactly $k$ products.
We can directly link the design $y$ to $w$ through the following constraints
\begin{subequations}\label{eq:endo_supply_reformulation}
    \begin{align}
        &\sum_{j \in J} y_{i,j} = \sum_{k=0}^{|J|} k \; w_{i,k} &&\forall i \in I, \label{eq:endo_supply_a} \\
        &\sum_{k=0}^{|J|} w_{i,k} = 1 &&\forall i \in I. \label{eq:endo_supply_b}
    \end{align}
\end{subequations}
Using the $w$ variables, we can replace the constraints in~\eqref{eq:no_good_cut} by 
\begin{equation}\label{eq:endo_supply_no_good_cut}
    \mu \le \left[ 
        \sum_{i \in I} (1 - w_{i,d_i})
    \right] \cdot U + F(y, \mathbb{P}_{d}) \quad \forall d \in \{0, \dots, |J|\}^{I}.
\end{equation}
Notice that the number of inequalities in~\eqref{eq:endo_supply_no_good_cut} is $(|J| + 1)^{|I|}$, which is significantly smaller than the number of inequalities in~\eqref{eq:no_good_cut}.  

\subsection{Reformulation of Endogenous Demand Uncertainty}\label{sec:endo_demand_formulation}

In analogy to the previous section, suppose now that supply is exogenous.
Recall that in the case of endogenous demand uncertainty, we consider a partition $\mathcal{Z}$ of the plants $I$ into zones.
Note that a design $y$ induces an assignment $h : J \to 2^{\mathcal{Z}}$ that maps each product $j \in J$ to the subset of zones from which it is sourced.
The demand distribution is determined by this assignment $h$, and hence we denote it by $\mathbb{P}_h$.
For each $j \in J$ and $z \in \mathcal{Z}$, we introduce a binary variable $u_{j,z} \in \{0,1\}$ indicating whether product $j$ is assigned to a plant in zone $z$.
The design $y$ can be linked to $u$ through the following constraints:
\begin{subequations}\label{eq:link_y_u}
    \begin{align}
        &y_{i,j} \le u_{j,z} &&\forall j \in J, \, z \in \mathcal{Z}, i \in z, \label{eq:endo_demand_a}\\
        &u_{j,z} \le \sum_{i \in z} y_{i,j}  &&\forall z \in \mathcal{Z}, \, j \in J. \label{eq:endo_demand_b}
    \end{align}
\end{subequations}
Similar to the constraints in~\eqref{eq:endo_supply_no_good_cut}, we may replace the constraints in~\eqref{eq:no_good_cut} by
\begin{equation}\label{eq:endo_demand_no_good_cut}
    \mu \le \left[\sum_{j \in J} \left(\sum  \nolimits_{z \in h(j)} ( 1-u_{j,z}) + \sum \nolimits_{z \in \mathcal{Z} \setminus h(j)} u_{j,z} \right) \right] \cdot U + F(y, \mathbb{P}_h) \quad \forall h:J \to 2^{\mathcal{Z}}.
\end{equation}
The number of inequalities in~\eqref{eq:endo_demand_no_good_cut} is $2^{|J|\cdot|\mathcal{Z}|}$.
Assuming that there are fewer zones than plants, this is again smaller than the number of inequalities in~\eqref{eq:no_good_cut}. 

\subsection{Combining Endogenous Supply and Demand}
When both supply and demand are endogenous, the distribution $\mathbb{P}_y$ is determined by both the degree vector of plants and the assignment of products to zones.
Hence, for a degree vector $d \in \mathbb{Z}_{\ge 0}^I$ and an assignment $h : J \to 2^{\mathcal{Z}}$, we let $\mathbb{P}_{d,h}$ denote the distribution that is induced when each plant $i$ is assigned to exactly $d_i$ products and each product $j$ is assigned to zones in $h(j)$.
We can combine the two previous reformulations in a straightforward way to finally obtain an aggregation of the constraints in~\eqref{eq:no_good_cut}.
The resulting formulation is  
\begin{subequations}\label{eq:model}
    \begin{align}
        \max_{y, w, u, \mu} \ &\mu - \sum_{i\in I} \sum_{j \in J} s_{i,j} \; y_{i,j} \label{eq:combined_objective} \\ 
        \text{s.t.} \ 
        &\text{Constraints~\eqref{eq:endo_supply_reformulation}, \eqref{eq:link_y_u}},\nonumber \\
        &\mu \le \left[ 
            \sum_{i \in I} (1 - w_{i,d_i}) + \sum_{j \in J} \left(\sum  \nolimits_{z \in h(j)} ( 1-u_{j,z}) + \sum \nolimits_{z \in \mathcal{Z} \setminus h(j)} u_{j,z} \right) 
        \right] \cdot U + F(y, \mathbb{P}_{d,h}) \nonumber \\
        &\qquad \forall d \in \{0,\dots,|J|\}^{I}, \;
        h : J \to 2^{\mathcal{Z}}, \label{eq:combined_nogood_cuts} \\[0.5em]
        & y \in \{0,1\}^{I\times J},\;
        w \in \{0,1\}^{I\times \{0,\dots,|J|\}},\;
        u \in \{0,1\}^{J\times \mathcal{Z}}.\label{eq:domain}
    \end{align}
\end{subequations}

\subsection{Distribution-specific Constraints}
In the final step, we dualize the term $F(y, \mathbb{P}_{d,h})$ in Constraints~\eqref{eq:combined_nogood_cuts}.
Recall that $F(y,p) = \nicefrac{1}{|S_p|} \sum_{\xi \in S_p} f(y, \xi)$.
Each second-stage problem $f(y, \xi)$ is a linear program (which is always feasible and finite), and hence we may consider its dual formulation
\begin{equation}\label{eq:dual}
        f(y, \xi) = \min_{\alpha, \beta, \rho} \ \left\{ \sum_{i \in I} \xi_{i}^{c} \; \alpha_{i} + \sum_{j \in J} \xi_{j}^{d} \; \beta_j + \sum_{(i,j) \in I \times J} M_{i,j}^\xi \; y_{i,j} \; \rho_{i,j} \ | \ (\alpha, \beta, \rho) \in \mathcal{T} \right\},
\end{equation}
where $\mathcal{T} \coloneq \left\{(\alpha, \beta, \rho) \in \mathbb{R}^I_{\ge 0} \times \mathbb{R}^J_{\ge 0} \times \mathbb{R}^{I \times J}_{\ge 0}:  r_{i,j} \; \alpha_i + \beta_j +  \rho_{i,j} \ge q_{i,j} \quad \forall i \in I, j \in J \right\}$.
Note that $\mathcal{T}$ is a non-empty polyhedron and independent of $\xi$.
Denoting the vertices of $\mathcal{T}$ by $\vertices(\mathcal{T})$, we can restrict the minimization in Problem~\eqref{eq:dual} to $(\alpha,\beta,\rho) \in \vertices(\mathcal{T})$.
Setting $S_{d,h} \coloneq S_{\mathbb{P}_{d,h}}$, we can  reformulate Constraints~\eqref{eq:combined_nogood_cuts} as
\begin{align}
    \mu \le \ &\left[ 
        \sum_{i \in I} (1 - w_{i,d_i}) + \sum_{j \in J} \left(\sum  \nolimits_{z \in h(j)} ( 1-u_{j,z}) + \sum \nolimits_{z \in \mathcal{Z} \setminus h(j)} u_{j,z} \right) 
    \right] \cdot U \nonumber \\
    &+ \frac{1}{|S_{d,h}|} \sum_{\xi \in S_{d,h}} \left( \sum_{i \in I} \xi_{i}^{c} \; \alpha_{i}^{\xi} + \sum_{j \in J} \xi_{j}^{d} \; \beta_j^{\xi} + \sum_{(i,j) \in I \times J} M_{i,j}^\xi  \; \rho_{i,j}^{\xi} \; y_{i,j} \right) \nonumber \\
    &\quad \forall d \in \{0,\dots,|J|\}^{I}, \;
    h : J \to 2^{\mathcal{Z}}, \;
    (\alpha, \beta, \rho) \in \vertices(\mathcal{T})^{S_{d,h}}. \label{eq:final_benders_cuts}\tag{P-Constr.}
\end{align}
We have now transformed the original Problem~\eqref{eq:generic_model} with endogenous uncertainty into a mixed-integer linear program.
Note that the number of constraints in~\eqref{eq:final_benders_cuts} is very large.
We refer to them as \emph{distribution-specific constraints} and tackle them in a lazy fashion as explained in Section~\ref{sec:solution_approach}.

%% file: strengthening.tex
\section{Formulation Strengthening}\label{sec:strengthening}
The Constraints~\eqref{eq:final_benders_cuts} in our formulation are distribution-specific, i.e., they only provide a valid upper bound on the expected second-stage profit for all first-stage designs $y$ that induce the \emph{specific} distribution $\mathbb{P}_{d,h}$.
This implies that every distribution $\mathbb{P}_{d,h}$ whose corresponding constraints have not yet been added to the formulation will allow the expected second-stage profit $\mu$ to be $+\infty$.
Therefore, the current formulation does not provide meaningful upper bounds until \emph{all} distributions have been \emph{visited} at least once.
In other words, any iterative approach based on the above formulation must visit all distributions at least once to guarantee convergence to optimality.

In what follows, we describe strengthened formulations that yield non-trivial upper bounds that are valid for all distributions or subsets of them.
We begin by adapting state-of-the-art inequalities from the literature and afterwards present our new strengthened formulations.

\subsection{State-of-the-art Cuts}\label{sec:sota_cuts}
 
Since our second-stage problem $f(y, \xi)$ is concave in $\xi$ for fixed $y$, the objective value of the second-stage problem can be used to derive a single inequality imposing a valid upper bound on $\mu$ for any design $y$ independent of distribution $p \in \mathcal{P}$.
This observation has been made in \cite[Proposition~4]{pantuso2025solutionstochasticfacilitylocation}:

\begin{proposition}\label{prop:jensen}
    Every design $y \in \{0,1\}^{I \times J}$ satisfies
    \[
        \mathbb{E}_{\Xi \sim \mathbb{P}_y}[f(y, \Xi)]
        \le \sum_{i\in I} \sum_{j \in J} q_{i,j} \cdot \max_{p \in \mathcal{P}} \ \min \left(
        \frac{\mathbb{E}_{\Xi \sim p} \left[\Xi_i^{c}\right]}{r_{i,j}}, \mathbb{E}_{\Xi \sim p} \left[\Xi_j^{d}\right]
        \right) \cdot y_{i,j}.
    \]
\end{proposition}
\begin{proof}{Proof.}
    Note that Problem~\eqref{eq:max_flow} includes the constraints $x_{i,j} \le M_{i,j}^\xi \; y_{i,j} = \min \left( \frac{\xi_i^{c}}{r_{i,j}}, \xi_j^{d} \right) \cdot y_{i,j}$.
    Hence, for fixed $y$ and $\xi$, we have
    $
        f(y,\xi) \le \sum_{i\in I} \sum_{j \in J} q_{i,j} \cdot \min (
        \frac{\xi_i^{c}}{r_{i,j}}, \xi_j^{d}) \cdot y_{i,j}.
    $
    Thus, by Jensen's inequality, we conclude
    \begin{align*}
        \mathbb{E}_{\Xi \sim \mathbb{P}_y}[f(y, \Xi)]
        \le f(y, \mathbb{E}_{\Xi \sim \mathbb{P}_y}[\Xi]) &
        \le \sum_{i\in I} \sum_{j \in J} q_{i,j} \cdot \ \min \left(
        \frac{\mathbb{E}_{\Xi \sim \mathbb{P}_y} \left[\Xi_i^{c}\right]}{r_{i,j}}, \mathbb{E}_{\Xi \sim \mathbb{P}_y} \left[\Xi_j^{d}\right] \right) \cdot y_{i,j} \\
        & \le \sum_{i\in I} \sum_{j \in J} \ q_{i,j} \cdot \max_{p \in \mathcal{P}} \ \min \left(
        \frac{\mathbb{E}_{\Xi \sim p} \left[\Xi_i^{c}\right]}{r_{i,j}}, \mathbb{E}_{\Xi \sim p} \left[\Xi_j^{d}\right]
        \right) \cdot y_{i,j}. \qedhere
    \end{align*}
\end{proof}
By the above statement, we may add the inequality
    \begin{equation}\label{eq:jensen}\tag{JC}
    \mu \le \sum_{i\in I} \sum_{j \in J} q_{i,j} \cdot \max_{p \in \mathcal{P}} \ \min \left(
        \frac{\mathbb{E}_{\Xi \sim p} \left[\Xi_i^{c}\right]}{r_{i,j}}, \mathbb{E}_{\Xi \sim p} \left[\Xi_j^{d}\right]
        \right) \cdot y_{i,j}
\end{equation}
to our model, which we call the \emph{Jensen constraint}.
We may further use Proposition~\ref{prop:jensen} to assign a reasonable value to parameter $U$ in Problem~\eqref{eq:model}.
\begin{corollary}\label{cor:U}
    For Problem~\eqref{eq:model}, a valid choice of parameter $U$ is
    \[
        U = \sum_{i\in I} \sum_{j \in J} q_{i,j} \cdot \max_{p \in \mathcal{P}} \ \min \left(
        \frac{\mathbb{E}_{\Xi \sim p} \left[\Xi_i^{c}\right]}{r_{i,j}}, \mathbb{E}_{\Xi \sim p} \left[\Xi_j^{d}\right]
        \right).
    \]
\end{corollary}

Furthermore, our second-stage problem $f(y, \xi)$ is non-decreasing in $\xi$ for fixed $y$.
This allows us to derive a valid upper bound for all distributions using the so-called dominant scenario \citep{pantuso2025shaped}.
Here, the \emph{dominant scenario} is the vector $\hat{\xi} = \max_{p \in \mathcal{P}} \mathbb{E}_{\Xi \sim p}[\Xi]$, i.e., the component-wise maximum expected supply and demand over all distributions.
Note that it satisfies
\begin{equation}\label{eq:gsc1}
    f(y, \mathbb{E}_{\Xi \sim \mathbb{P}_y}[\Xi]) \le f(y, \hat{\xi})
\end{equation}
for every design $y$, and that we do not require $\hat{\xi}$ to be among the sampled scenarios.
As observed in \cite[Proposition 4]{pantuso2025shaped}, this allows us to derive another valid inequality.

\begin{proposition}\label{prop:dominant_scenario}
    Let $\hat{\xi}$ be the dominant scenario and $(\alpha, \beta, \rho) \in \vertices(\mathcal{T})$.
    Every design $y \in \{0,1\}^{I \times J}$ satisfies
    \[
        \mathbb{E}_{\Xi \sim \mathbb{P}_{y}} [f(y, \Xi)] \le \sum_{i \in I} \hat{\xi}_{i}^{c} \; \alpha_{i} + \sum_{j \in J} \hat{\xi}_{j}^{d} \; \beta_{j} + \sum_{(i,j) \in I \times J} M_{i,j}^{\hat{\xi}} \; \rho_{i,j} \; y_{i,j}.
    \]
\end{proposition}
\begin{proof}{Proof.}
    We have
    \begin{equation*}
        \mathbb{E}_{\Xi \sim \mathbb{P}_{y}} [f(y, \Xi)] \le f(y, \mathbb{E}_{\Xi \sim \mathbb{P}_y}[\Xi]) \le f(y, \hat{\xi}) \;
        \le \sum_{i \in I} \hat{\xi}_{i}^{c} \; \alpha_{i} + \sum_{j \in J} \hat{\xi}_{j}^{d} \; \beta_{j} + \sum_{(i,j) \in I \times J} M_{i,j}^{\hat{\xi}} \; \rho_{i,j} \; y_{i,j},
    \end{equation*}
    where the second inequality follows from~\eqref{eq:gsc1} and the third inequality follows from Problem~\eqref{eq:dual}.\qedhere
\end{proof}
As a result, the \emph{dominant-scenario cut}
\begin{equation}\label{eq:dominant_scenario}\tag{DC}
    \mu \le \sum_{i \in I} \hat{\xi}_{i}^{c} \; \alpha_{i} + \sum_{j \in J} \hat{\xi}_{j}^{d} \; \beta_{j} + \sum_{(i,j) \in I \times J} M_{i,j}^{\hat{\xi}} \; \rho_{i,j} \; y_{i,j}
\end{equation}
is valid for our model.

\subsection{Dominant Flow Constraints}\label{sec:our_cuts}

Since the dominant scenario $\hat{\xi}$ satisfies $\mathbb{E}_{\Xi \sim \mathbb{P}_{y}} [f(y, \Xi)] \le f(y, \hat{\xi})$ for all designs $y$, the inequality $\mu \le f(y, \hat{\xi})$ is valid for Problem~\eqref{eq:model}.
We propose to model this inequality by introducing second-stage flow variables $\hat{x}_{i,j}$ that represent flows under the dominant scenario $\hat{\xi}$ and adding the \emph{dominant flow constraints}
\begin{equation}\label{eq:dff}\tag{DFC}
    \begin{aligned}
        &\; \mu \le \sum_{i\in I} \sum_{j \in J} q_{i,j} \; \hat{x}_{i,j},\\
        &\sum_{j \in J}   r_{i,j} \; \hat{x}_{i,j} \le \hat{\xi}_{i}^{c} &&\forall \ i \in I, \\
        & \sum_{i \in I} \hat{x}_{i,j} \le \hat{\xi}_{j}^{d} &&\forall \ j \in J, \\
        & \; 0 \le \hat{x}_{i,j} \le M_{i,j}^{\hat{\xi}} \; y_{i,j} &&\forall \ i \in I, \, j \in J,
    \end{aligned}
\end{equation}
to our model.
Note that the above constraints impose a valid upper bound on $\mu$ since for every $y$, there exists a feasible flow $\hat{x}$ such that $f(y, \hat{\xi}) = \sum_{i\in I} \sum_{j \in J} q_{i,j} \; \hat{x}_{i,j}$.

\begin{proposition}\label{prop:imply}
    \eqref{eq:dff} implies \eqref{eq:dominant_scenario}.
\end{proposition}
\begin{proof}{Proof.}
    If $(\mu, y, \hat{x})$ satisfies~\eqref{eq:dff} and $(\alpha, \beta, \rho) \in \vertices(\mathcal{T})$, then we have
    \[
        \mu \le \sum_{i \in I} \sum_{j \in J} q_{i,j} \; \hat{x}_{i,j}
        \le f(y, \hat{\xi})
        \le \sum_{i \in I} \hat{\xi}_{i}^{c} \; \alpha_{i} + \sum_{j \in J} \hat{\xi}_{j}^{d} \; \beta_{j} + \sum_{(i,j) \in I \times J} M_{i,j}^{\hat{\xi}} \; \rho_{i,j} \; y_{i,j},
    \]
    where the second inequality follows from the definition of $f(y, \hat{\xi})$ and the third inequality follows from Problem~\eqref{eq:dual}.\qedhere
\end{proof}

\begin{proposition}\label{prop:imply2}
    If $\Xi^c$ and $\Xi^d$ are independent, then
    \eqref{eq:dff} implies \eqref{eq:jensen}.
\end{proposition}
\begin{proof}{Proof.}
    By independence of $\Xi^c$ and $\Xi^d$, for every $i \in I$ and $j \in J$, there exists a distribution $p^{i,j} \in \mathcal{P}$ such that
    \begin{equation*}
        \label{eq:independence}
        \hat{\xi}_i^c = \max_{p \in \mathcal{P}} \mathbb{E}_{\Xi \sim p} \left[\Xi_i^{c}\right] = \mathbb{E}_{\Xi \sim p^{i,j}} \left[\Xi_i^{c}\right]
        \quad \text{and} \quad
        \hat{\xi}_j^d = \max_{p \in \mathcal{P}} \mathbb{E}_{\Xi \sim p} \left[\Xi_j^{d}\right] = \mathbb{E}_{\Xi \sim p^{i,j}} \left[\Xi_j^{d}\right].
    \end{equation*}
    If $(\mu, y, \hat{x})$ satisfies~\eqref{eq:dff}, then
    \begin{align*}
        \mu \le \sum_{i \in I} \sum_{j \in J} q_{i,j} \hat{x}_{i,j}
        & \le \sum_{i\in I} \sum_{j \in J} q_{i,j} \cdot M_{i,j}^{\hat{\xi}} \cdot y_{i,j}
        = \sum_{i\in I} \sum_{j \in J} q_{i,j} \cdot \min \left( \frac{\hat{\xi}_i^c}{r_{i,j}}, \hat{\xi}_j^d \right) \cdot y_{i,j} \\
        & \qquad = \sum_{i\in I} \sum_{j \in J} q_{i,j} \cdot \min \left(
        \frac{\mathbb{E}_{\Xi \sim p^{i,j}} \left[\Xi_i^{c}\right]}{r_{i,j}}, \mathbb{E}_{\Xi \sim p^{i,j}} \left[\Xi_j^{d}\right] \right) \cdot y_{i,j} \\
        & \qquad \le \sum_{i\in I} \sum_{j \in J} q_{i,j} \cdot \max_{p \in \mathcal{P}} \min \left( \frac{\mathbb{E}_{\Xi \sim p} \left[\Xi_i^{c}\right]}{r_{i,j}}
        , \mathbb{E}_{\Xi \sim p} \left[\Xi_j^{d}\right]
        \right) \cdot y_{i,j}. \qedhere
    \end{align*}
\end{proof}
We remark that in our application, supply and demand distributions are naturally independent.

\subsection{Dominant Flow Constraints for Subsets of Distributions}\label{sec:refined_dff}

While \eqref{eq:jensen}, \eqref{eq:dominant_scenario}, and~\eqref{eq:dff} impose valid upper bounds for \emph{all} distributions in $\mathcal{P}$, we can further strengthen our formulation by imposing valid upper bounds on subsets of distributions.
To illustrate the idea, let $i' \in I$ and $k' \in \{2, \dots, |J|\}$ (at $k'= 1$, the dominant flow constraints for all distributions are already valid because the highest performance of a plant is achieved when it is dedicated to one product).
Denote by $\mathcal{P}_{i', k'} \subseteq \mathcal{P}$ the subset of distributions induced by designs in which plant node $i'$ has a degree of at least $k'$.
Now, we may refine the dominant scenario over $\mathcal{P}_{i', k'}$ as $\hat{\xi}(i', k') \coloneqq \max_{p \in \mathcal{P}_{i', k'}} \mathbb{E}_{\Xi \sim p}[\Xi]$.
Then, the inequality $\mu \le f(y, \hat{\xi}(i',k'))$ is valid for all designs $y$ that induce distributions in $\mathcal{P}_{i', k'}$.
Hence, the constraint
    $\mu \le f(y, \hat{\xi}(i',k')) + \sum_{k = 0}^{k'-1} w_{i', k} \cdot U$
is valid for our formulation.
Again, we replace $f(y, \hat{\xi}(i',k'))$ by introducing second-stage flow variables $\hat{x}(i',k')$ that represent flows under the dominant scenario $\hat{\xi}(i', k')$.
The resulting constraints are
\begin{equation}\label{eq:dffs}\tag{DFC-S}
    \begin{aligned}
        &\; \mu \le \sum_{i\in I} \sum_{j \in J} q_{i,j} \; \hat{x}(i',k')_{i,j} + \sum_{k = 0}^{k'-1} w_{i', k} \cdot U,\\
        &\sum_{j \in J}   r_{i,j} \; \hat{x}(i',k')_{i,j} \le \hat{\xi}(i',k')^c_i &&\forall \ i \in I, \\
        & \sum_{i \in I} \hat{x}(i',k')_{i,j} \le \hat{\xi}(i',k')^d_j &&\forall \ j \in J, \\
        & \; 0 \le \hat{x}(i',k')_{i,j} \le M_{i,j}^{\hat{\xi}(i',k')} \; y_{i,j} &&\forall \ i \in I, \, j \in J.
    \end{aligned}
\end{equation}
Later, when we refer to \eqref{eq:dffs}, we mean including these constraints for all choices of $i' \in I$ and $k' \in \{2, \dots, |J|\}$.

Another natural refinement of the dominant scenario arises from considering subsets of demand distributions as follows.
Let $j' \in J$ and $\emptyset \ne Z \subseteq \mathcal{Z}$, and let $\mathcal{P}_{j', Z} \subseteq \mathcal{P}$ be the subset of distributions that are induced by designs such that product $j'$ is only assigned to plants in zones $Z$ and for each zone in $Z$ there is at least one plant assigned to product $j'$.
Analogously, we may refine the dominant scenario over $\mathcal{P}_{j', Z}$ and call it $\hat{\xi}(j', Z)$.
Consequently, the constraint
$\mu \le f(y, \hat{\xi}(j',Z)) + [ |Z| - \sum_{z \in Z} u_{j', z} + \sum_{z \notin Z}  u_{j', z}] \cdot U$
is valid for our formulation.
Introducing second-stage flow variables $\hat{x}(j',Z)$ that represent flows under the dominant scenario $\hat{\xi}(j', Z)$, we obtain the constraints
\begin{equation}\label{eq:dffd}\tag{DFC-D}
    \begin{aligned}
        &\; \mu \le \sum_{i\in I} \sum_{j \in J} q_{i,j} \; \hat{x}(j',Z)_{i,j} + \left[ |Z| - 
            \sum_{z \in Z} u_{j', z} + \sum_{z \notin Z} u_{j', z}
        \right] \cdot U,\\
        &\sum_{j \in J}   r_{i,j} \; \hat{x}(j',Z)_{i,j} \le \hat{\xi}(j',Z)^c_i &&\forall \ i \in I, \\
        & \sum_{i \in I} \hat{x}(j',Z)_{i,j} \le \hat{\xi}(j',Z)^d_j &&\forall \ j \in J, \\
        & \; 0 \le \hat{x}(j',Z)_{i,j} \le M_{i,j}^{\hat{\xi}(j',Z)} \; y_{i,j} &&\forall \ i \in I, \, j \in J.
    \end{aligned}
\end{equation}
When we refer to \eqref{eq:dffd}, we mean including these constraints for all choices of $j' \in J$ and $\emptyset \ne Z \subseteq \mathcal{Z}$.
We exclude $Z=\emptyset$ because this case can provide a nontrivial bound on $\mu$ only when $y_{i,j'}=0$ for all $i\in I$, in which case the resulting constraints do not further strengthen \eqref{eq:dff}.

Clearly, one may include further refinements of dominant scenarios by considering other subsets of distributions.
In this work, we focus on the above two types of refinements since they are natural in our application context, i.e., the grouping of distributions into their corresponding subsets aligns with the definition and use of the auxiliary variables $w$ and $u$.
To include~\eqref{eq:dffs} and~\eqref{eq:dffd} in Problem~\eqref{eq:model}, we need $(|J|-1)\cdot|I| + 2^{|\mathcal{Z}|} \cdot|J|$ flow formulations, each having $|I|\times|J|$ flow variables and $1 + |I| + |J| + |I| \times |J| $ constraints.
Therefore, it is important to assess the trade-off between the computational benefits of adding these constraints and the additional computational burden they impose.
We investigate this trade-off in our computational study in Section~\ref{sec:experiments}.

%% file: solution.tex
\section{Solution Approach}\label{sec:solution_approach}

In Section~\ref{sec:model}, we presented a mixed-integer linear programming formulation for the network flexibility design problem under endogenous uncertainty.
Afterwards, we derived valid distribution-free inequalities in Section~\ref{sec:strengthening} to strengthen the formulation.
Next, we adopt the L-shaped method for two-stage stochastic programs with endogenous uncertainty \citep{pantuso2025shaped} to solve the problem.
We remove all distribution-specific constraints~\eqref{eq:final_benders_cuts} from Problem~\eqref{eq:model} to obtain a relaxed master problem
\begin{equation}\label{eq:rmp}\tag{RMP}
    \max_{y, \mu, w, u} \quad \{\eqref{eq:combined_objective} \ | \ \eqref{eq:endo_supply_reformulation}, \ \eqref{eq:link_y_u}, \ \eqref{eq:domain}\},
\end{equation} 
and then iteratively add optimality cuts from~\eqref{eq:final_benders_cuts} to~\eqref{eq:rmp} until its solution is feasible to Problem~\eqref{eq:model}.
Note that no feasibility cuts are required since the second-stage problem is feasible for any first-stage solution $y$.
Our decomposition scheme with distribution-specific optimality cuts is summarized in Algorithm~\ref{alg:decomposition}.
\begin{algorithm}
    \caption{Decomposition scheme with distribution-specific optimality cuts}
    \label{alg:decomposition}
    \begin{algorithmic}[1]
        \State \textbf{Input:} \eqref{eq:rmp}.
        \State \textbf{Output:} Optimal design $y^*$.
        \While{true}
            \State Solve~\eqref{eq:rmp} to obtain $(y^{*}, \mu^{*}, w^{*}, u^{*})$.
            \State Determine $\mathbb{P}_{d,h}$ where $d$ and $h$ correspond to $w^{*}$ and $u^{*}$, respectively.
            \State Compute $(\alpha^{*, \xi}, \beta^{*, \xi}, \rho^{*, \xi})$ for all $\xi \in S_{d,h}$ by solving Problem~\eqref{eq:dual} with $y = y^*$.
            \If{$\mu^* \le \frac{1}{|S_{d,h}|} \sum_{\xi \in S_{d,h}} \left( \sum_{i \in I} \xi_{i}^{c} \; \alpha_{i}^{*, \xi} + \sum_{j \in J} \xi_{j}^{d} \; \beta_j^{*, \xi} + \sum_{(i,j) \in I \times J} M^{\xi}_{i,j}  \; \rho_{i,j}^{*, \xi} \; y^*_{i,j} \right)$}
                \State Return $y^*$.
            \Else
                \State Add~\eqref{eq:final_benders_cuts} for $(\alpha^{*}, \beta^{*}, \rho^{*})$, $d$, and $h$ to~\eqref{eq:rmp}.
            \EndIf
        \EndWhile
    \end{algorithmic}
\end{algorithm}
It serves as our baseline approach and can be strengthened with the formulations proposed in Section~\ref{sec:strengthening}.
We add~\eqref{eq:jensen} to~\eqref{eq:rmp} and include one~\eqref{eq:dominant_scenario} per iteration at Step~10 of Algorithm~\ref{alg:decomposition} (using $y^*$ and $\hat{\xi}$ in Problem~\eqref{eq:dual} and computing the corresponding optimal dual solution).
Our distribution-free formulations~\eqref{eq:dff},~\eqref{eq:dffs}, and~\eqref{eq:dffd} are added to~\eqref{eq:rmp} at the beginning of Algorithm~\ref{alg:decomposition}.

For our computational study (see Section~\ref{sec:cut_perf}), we consider the baseline approach and its strengthened versions, referred to as Approaches~(i)--(vi):
(i) baseline, (ii) \eqref{eq:jensen} and~\eqref{eq:dominant_scenario} (the state-of-the-art benchmark), (iii) \eqref{eq:dff}, (iv) \eqref{eq:dff},~\eqref{eq:dffs}, and~\eqref{eq:dffd},  where~\eqref{eq:dffs} is applied only when supply uncertainty is endogenous and~\eqref{eq:dffd} is applied only when demand uncertainty is endogenous, (v) \eqref{eq:dff} and~\eqref{eq:dffd}, (vi) \eqref{eq:dff} and~\eqref{eq:dffs} (see Table~\ref{tab:constraint_configurations}).

\begin{table}[t]
    \centering
    \small
    \setlength{\tabcolsep}{9pt}
    \renewcommand{\arraystretch}{1}
    \caption{Strengthened formulation variations for our computational study.}
    \label{tab:constraint_configurations}
    \begin{tabular}{llllll}
        \toprule
        (i) & (ii) & (iii) & (iv) & (v) & (vi) \\
        \midrule

        \begin{tabular}[t]{@{}l@{}}
            baseline
        \end{tabular}
        &
        \begin{tabular}[t]{@{}l@{}}
            \eqref{eq:jensen} \\
            \eqref{eq:dominant_scenario}
        \end{tabular}
        &
        \begin{tabular}[t]{@{}l@{}}
            \eqref{eq:dff}
        \end{tabular}
        &
        \begin{tabular}[t]{@{}l@{}}
            \eqref{eq:dff} \\
            \eqref{eq:dffs} \\
            \eqref{eq:dffd}
        \end{tabular}
        &
        \begin{tabular}[t]{@{}l@{}}
            \eqref{eq:dff} \\
            \eqref{eq:dffd}
        \end{tabular}
        &
        \begin{tabular}[t]{@{}l@{}}
            \eqref{eq:dff} \\
            \eqref{eq:dffs}
        \end{tabular}
        \\

        \bottomrule
    \end{tabular}
\end{table}

For the sake of completeness, we present formal statements showing that optimality cuts derived at Step~10 of Algorithm~\ref{alg:decomposition} remove infeasible solutions from~\eqref{eq:rmp} and that Algorithm~\ref{alg:decomposition} converges in finitely many iterations (see Appendix~\ref{sec:appendix_proofs}).

%% file: experiments.tex
\section{Computational Insights}\label{sec:experiments}

In this section, we provide computational insights into the performance of our baseline approach and distribution-free formulations.
First, we compare the baseline approach with the exhaustive distribution enumeration approach (see Section~\ref{sec:naive}), serving as an exact benchmark for our problem.
Second, we evaluate the performance of the baseline approach when augmented with the formulations proposed in Section~\ref{sec:strengthening}.
We assess the effectiveness of Approaches~(i)--(iv) in three uncertainty regimes: Regime~(1), endogenous supply and demand; Regime~(2), exogenous supply and endogenous demand; and Regime~(3), endogenous supply and exogenous demand.
This allows us to assess the benefits of the proposed strengthened formulations under various uncertainty regimes.
Moreover, we examine the relative impact of different combinations of the distribution-free formulations, namely, Approaches~(iii)--(vi), under varying parameter settings in Regime~(1).
Finally, we demonstrate that Approach~(iv) provides strong lower bounds in the early stages of the solution process.

\subsection{Setup of Experiments}\label{sec:setup_expo}

We generate instances of 12 sizes with 2, 3, and 4 plants, and the number of products ranging from the number of plants to twice the number of plants.
Instance sizes are selected such that problems remain computationally tractable within the time limit of $1{,}800$ seconds.
We denote each instance size by a tuple $(|I|,|J|)$ and report average results over 10 random instances per size.

For all instances, we define $r_{i,j}$, $s_{i,j}$, and $q_{i,j}$ as 1, 100, and 7, respectively (our choice of parameters is in line with \cite{feng2017process}).
We assume that, for every design choice $y$, plants' capacities and products' demands are independent and normally distributed.
Each instance is further specified using distribution parameters $\mu_i^c$, $\sigma_i^c$ for each $i \in I$ and $\mu_j^d$, $\sigma_j^d$ for each $j \in J$, which will define the distribution $\mathbb{P}_{d,h}$ (to be explained in the next paragraph).
We set $\mu_j^d = 100$ for all $j \in J$, and sample $\mu^c \in \R^I$ uniformly from the polytope
$
    \Gamma \coloneq \{ \mu \in \R^{I} : \mu_{i} \geq 110 \; \forall i \in I, \, 110 \cdot |J| \le \sum \nolimits_{i \in I} \mu_{i} \le 130 \cdot |J|\}.
$
Moreover, we sample $\sigma_i^c$ and $\sigma_j^d$ uniformly from the interval $(0.05 \mu_i^c, 0.15 \mu_i^c)$ and the interval $(0.3 \mu_j^d, 0.5 \mu_j^d)$, respectively.
The choice of $\sigma_j^d$ is in the range of the $40\%$ forecast error typically observed in the automotive industry (see, e.g., \citeauthor{jordan1995principles}, \citeyear{jordan1995principles}; \citeauthor{deng2013process}, \citeyear{deng2013process}).
This construction yields heterogeneous and unbalanced production networks in which total expected supply moderately exceeds total expected demand, 
reducing the likelihood that plants or products are excluded due to capacity scarcity.
In scenario sampling, negative realizations are truncated at zero. 
To keep the samples centered around their nominal mean values, realizations exceeding twice the mean are truncated at twice the mean.

The distribution $\mathbb{P}_{d,h}$ is defined as follows.
On the supply side, if plant $i$ is assigned to $d_{i}$ products, then its capacity follows a normal distribution with mean $\mu^c_{i,d_i} \coloneq \mu^c_{i} \cdot (1 - 0.0162 \cdot (d_{i} - 1))$ and standard deviation $\sigma^c_{i,d_i} \coloneq \nicefrac{\sigma_i^c}{\mu_i^c} \cdot \mu^c_{i,d_i} $.
\citet{choudhary2018flexibility} report that automotive assembly plants on average lose approximately $1.62\%$ of annual capacity due to model changeovers. 
In this definition, the supply coefficient of variation (CV) remains unchanged.

On the demand side, we generate two zones $\mathcal{Z} = \{0,1\}$ (namely, domestic and foreign zones).
To assign plants to zones, we sample zone assignments uniformly from the set of all possible assignments, in which each zone contains at least one plant. 
If product $j$ is only assigned to the domestic zone, i.e., $h(j) = \{0\}$, then $\mu_{j, h(j)}^d = \mu_j^d$ and $\sigma_{j, h(j)}^d = \sigma^d_j$.
In case it is assigned only to the foreign zone, i.e., $h(j) = \{1\}$, the mean demand reduces to $\mu_{j, h(j)}^d = \mu_j^d \cdot (1 - 0.3)$ and the demand CV increases such that $\sigma_{j, h(j)}^d = (1+\nicefrac{0.3}{3}) \cdot \nicefrac{\sigma_j^d}{\mu_j^d} \cdot \mu_{j, h(j)}^d $.
Regarding the choice of parameters, the $30\%$ mean demand loss is based on the import reductions reported by \cite{amiti2019impact}, and the $10\%$ CV increase is a conservative estimate to reflect the heightened uncertainty level associated with foreign sourcing.
Our mechanism for demand loss due to foreign sourcing aligns with our country-of-origin motivations as well.
In the case of mixed sourcing for product $j$, i.e., $h(j) = \{0,1\}$, we set $\mu_{j, h(j)}^d = \mu_j^d \cdot (1 - 0.11)$ and $\sigma_{j, h(j)}^d = (1+\nicefrac{0.11}{3}) \cdot \nicefrac{\sigma_j^d}{\mu_j^d} \cdot \mu_{j, h(j)}^d$.

When supply or demand uncertainty is exogenous (in Regimes~(2) and~(3)), supply at plant $i$ follows a normal distribution with parameters $\mu_i^c$ and $\sigma_i^c$, and demand for product $j$ follows a normal distribution with parameters $\mu_j^d$ and $\sigma_j^d$.
Note that our design of experiments is for a focal market that aggregates the demand behavior of all markets given a sourcing decision, rather than modeling each market separately. 
Given the strategic nature of the supply network design problem, such aggregation is common in supply chain management. 

To implement the proposed formulations in Section~\ref{sec:strengthening}, we need to specify the dominant scenarios $\hat{\xi}$, $\hat{\xi}(i',k')$  for each $i' \in I$ and $k' \in \{2, \dots, |J|\}$, and $\hat{\xi}(j', Z)$ for each $j' \in J$ and $Z \subseteq \mathcal{Z} \setminus \{\emptyset\}$.
By our design of uncertainty mechanism for $\mathbb{P}_{d,h}$, $\hat{\xi} = (\mu^c, \mu^d)$.
Furthermore, $\hat{\xi}(i',k')$ is determined as $\hat{\xi}(i',k')_{i'}^c = \mu_{i',k'}^c$, $\hat{\xi}(i',k')_i^c = \mu_i^c$ for all $i \in I\setminus\{i'\}$, and $\hat{\xi}(i',k')_j^d = \mu_j^d$ for all $j \in J$.
Finally, for $\hat{\xi}(j', Z)$ we have $\hat{\xi}(j', Z)_{j'}^d = \mu_{j',Z}^d$, $\hat{\xi}(j', Z)_j^d = \mu_j^d$ for all $j \in J\setminus\{j'\}$, and $\hat{\xi}(j', Z)_i^c = \mu_i^c$ for all $i \in I$. 

All experiments are performed on a standard laptop.
The implementation is coded in Python, and Gurobi (version 13) is used for solving the optimization problems.
Our implementation of Algorithm~\ref{alg:decomposition} relies on Gurobi callback functions to add optimality cuts on integer nodes of the branch-and-bound tree.
At Step~10 of Algorithm~\ref{alg:decomposition}, we feed feasible solutions with corrected expected second-stage profit $\mu$ to the solver.

To approximate the second-stage problem, we set $|S_{d,h}| = 1{,}000$ for all $(d,h) \in \mathcal{P}$.
We found this sample size to be sufficiently large for Algorithm~\ref{alg:decomposition} to yield consistent optimal objective values, up to numerical precision, across most instances and approaches solved to optimality within the time limit.
Because the number of distributions is large, it is not efficient to pre-sample scenarios for every distribution in advance.
Instead, once an incumbent solution and its induced distribution are visited during the solution process, we sample scenarios.
Moreover, we ensure the consistency of sampled scenarios across different algorithmic settings.
Furthermore, we use five parallel processes to solve the scenario sub-problems concurrently. 
To uniformly sample mean capacities from $\Gamma$, we implement the hit-and-run approach~\citep{sun2024polytopewalksparsemcmcsampling} using the python package \texttt{PolytopeWalk}.

\subsection{Performance Against Exhaustive Distribution Enumeration}\label{sec:naive}

Recall that the baseline approach must enumerate all distributions in $\mathcal{P}$ at least once before it terminates.
This observation motivates exhaustive distribution enumeration (EDE) as a natural benchmark for our baseline model:
for all combinations of degree vectors $d \in \mathbb{Z}_{\ge 0}^I$ and assignments $h : J \to 2^{\mathcal{Z}}$, we fix the distribution $\mathbb{P}_{d,h}$ and solve a standard two-stage stochastic program over the subset of design choices $y \in \{0,1\}^{I \times J}$ corresponding to $d$ and $h$.
The best solution among all distributions is the optimal design (the mathematical model solved at each iteration of the EDE approach is available in Appendix~\ref{sec:appendix_model}).

We implement EDE using the standard L-shaped method with Gurobi callbacks and parallel scenario computations.
We consider Regime~(1) and solve instances of sizes $(2,2)$ to $(3,5)$.
Both approaches yield the same objective values for the same instance (up to numerical precision).
We compare them with respect to the solution time and the number of iterations during the solution process (for the EDE approach, we sum the total number of iterations of the L-shaped method over all two-stage stochastic programs corresponding to individual distributions).

Figure~\ref{fig:naive_vs_decomp} illustrates the ratio of the two performance metrics for EDE relative to the baseline approach, which is above one in both metrics across all instance sizes.
For instances of sizes $(3,4)$ and $(3,5)$, the baseline approach needs a much smaller number of iterations compared to the EDE approach.
Intuitively, the baseline approach needs fewer inequalities to exclude distributions that do not correspond to optimal solutions from the search space.
In contrast, the EDE approach has to solve a two-stage stochastic program to optimality for every distribution.   
Regarding instances of sizes $(2,2)$ to $(3,3)$, both approaches require a comparable number of iterations.
However, the baseline approach still remains substantially faster.
This is counter-intuitive, as the two-stage stochastic program solved in each iteration of the EDE approach is smaller than the~\eqref{eq:rmp} solved in each iteration of the baseline approach. 
One possible explanation is that in the EDE approach, the solver has to repeatedly conduct an independent branch-and-bound search for each distribution.
In contrast, in the baseline approach, the solver maintains one search process on the master problem.

\begin{figure}
\centering
\begin{tikzpicture}
\begin{axis}[
    ybar,
    width=0.6\textwidth,
    height=4cm,
    bar width=10pt,
    xtick=data,
    xticklabels={(2,2),(2,3),(2,4),(3,3),(3,4),(3,5)},
    ylabel={Ratio},
    ymin=1,
    enlarge x limits=0.25,
    grid=both,
    major grid style={dashed, gray!60},
    minor grid style={dashed, gray!30},
    legend style={
    at={(0.5,1.02)},
        anchor=south,
        font=\small,
        legend columns=2,
        column sep=8pt
    },
]

\addplot[
    fill={rgb,255:red,65; green,117; blue,5},
    pattern=north east lines,
    pattern color=black,
    draw=black,
    bar shift=-6pt,
]
coordinates {
(1, 1.588235294)
(2, 1.215384615)
(3, 1.070038911)
(4, 1.325416582)
(5, 2.438473864)
(6, 2.978934955)	
};
\addlegendentry{number of iterations}

\addplot[
    fill={rgb,255:red,111; green,111; blue,111},
    draw=black,
    bar shift=+6pt,
]
coordinates {
(1, 1.809045526)
(2, 2.056448104)
(3, 2.206412127)
(4, 2.024539491)
(5, 3.032842251)
(6, 3.496260571)
};
\addlegendentry{solution time}

\end{axis}
\end{tikzpicture}
\caption{Relative performance of the EDE approach to the baseline approach. 
Average values over 10 instances are reported.}
\label{fig:naive_vs_decomp}
\end{figure}
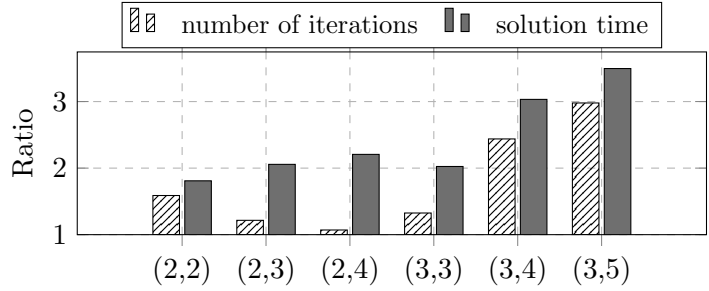

\subsection{Performance of Distribution-free Cuts}\label{sec:cut_perf}

Next, we compare the performance of Approaches~(i)--(iv), under Regimes~(1),~(2), and~(3).
Moreover, we compare Approaches~(iii)--(vi) under Regime~(1).
We report the optimality gap  
$\nicefrac{(\text{Upper Bound} - \text{Lower Bound})}{\text{Lower Bound}}$,
the number of iterations, the solution time, the number of visited distributions (VD), the ratio of visited supply distributions (RVSD), and the ratio of visited demand distributions (RVDD).
The latter two metrics are defined as the number of visited supply (demand) distributions divided by the total number of possible supply (demand) distributions.
Metrics VD, RVSD, and RVDD quantify the effectiveness of the formulations introduced in Section~\ref{sec:strengthening} in eliminating distributions from the search space.
For each instance size, we report the average metric values over the 10 random instances.

\subsubsection{Regime (1): Endogenous Supply and Demand Uncertainties}\label{sec:endoboth}
\input{endo_sup_endo_dem}

\subsubsection{Regime (2): Exogenous Supply and Endogenous Demand Uncertainties}\label{sec:endodem}
\input{exo_sup_endo_dem}

\subsubsection{Regime (3): Endogenous Supply and Exogenous Demand Uncertainties}\label{sec:endosup}
\input{endo_sup_exo_dem}

\subsection{Lower Bound Quality in Early Stages}\label{sec:early_lb}
For instances under Regime~(1), we observe that Approach~(iv) produces high-quality lower bounds early in the solution process.
To illustrate this behavior, Figure~\ref{fig:bounds_two_panels} shows the progress of the lower and upper bounds for two representative instances of sizes $(4,6)$ and $(4,7)$.
In both cases, the lower bound approaches the final objective value rapidly.
The final incumbent value is reached after approximately 157 and 265 seconds, respectively, while the lower bound is already within a $1\%$ gap of the final objective value after approximately one and four seconds, respectively.

\begin{figure}[H]
\centering

\begin{tikzpicture}
\begin{groupplot}[
    group style={group size=2 by 1, horizontal sep=3cm},
    width=0.4\textwidth,
    height=5cm,
    grid=both,
]

\nextgroupplot[
    title={Size (4,6)},
    xlabel={Time (s)},
    ylabel={},
    xmin=-15,                 
    xmax=1250,                
    xtick={0, 400, 800, 1200},
    xticklabels={0, 400, 800, 1200},
    scaled x ticks=false,
    enlarge x limits=false,
    legend columns=2,
    legend style={
        at={(1, 1)},
        anchor=south,
        xshift=1cm,
        yshift=0.7cm,
        draw=black,
        fill=white,
        fill opacity=0.95,
        text opacity=1,
        /tikz/every even column/.append style={column sep=10pt},
    }
]
\addplot+[const plot, color=black!75, mark=none, thick, dashed]
    table[x=t, y=ub, col sep=comma] {bounds_98.csv};
\addlegendentry{UB}

\addplot+[const plot, color=ORBlue!100, mark=none, thick]
    table[x=t, y=lb, col sep=comma] {bounds_98.csv};
\addlegendentry{LB}

\nextgroupplot[
    title={Size (4,7)},
    xlabel={Time (h)},
    ylabel={},
    ytick={3250, 3500, 3750, 4000, 4250},
    xmin=-600,
    xmax=30500,
    xtick={0, 7200, 14400, 21600, 28800},
    xticklabels={0,2,4,6,8},
    scaled x ticks=false,
    enlarge x limits=false,
]
\addplot+[const plot, color=black!75, mark=none, thick, dashed]
    table[x=t, y=ub, col sep=comma] {bounds_106.csv};

\addplot+[const plot, color=ORBlue!100, mark=none, thick]
    table[x=t, y=lb, col sep=comma] {bounds_106.csv};

\end{groupplot}
\end{tikzpicture}

\caption{Progress of the lower and upper bounds of Approach~(iv) for two representative instances (one for each size $(4,6)$ and $(4,7)$) under Regime~(1).}
\label{fig:bounds_two_panels}
\end{figure}

%% file: endo_sup_endo_dem.tex
Assuming supply and demand uncertainties are both endogenous, we first compare Approaches~(i)--(iv).
All approaches solve instances of sizes $(2,2)$ to $(4, 5)$ to optimality within the time limit, except for Approach~(i) for instances of sizes $(3,6)$ and $(4, 5)$.
For these solvable instances, Figure~\ref{fig:runboth} shows the solution times normalized by Approach~(iv).
For instances of sizes $(4,6)$, $(4,7)$, and $(4,8)$, Figure~\ref{fig:gapboth} shows the optimality gaps over all approaches, except for Approach~(i) as it fails to yield a meaningful upper bound.
Table~\ref{tab:both} in Appendix~\ref{sec:appendix_tables} summarizes the performance results for each instance size.

\begin{figure}
\centering
\pgfplotslegendfromname{sharedlegend}
\begin{subfigure}[b]{0.64\linewidth}
    \centering
    \begin{tikzpicture}
        \begin{axis}[
            width=\linewidth,
            height=5.5cm,
            ymode=log,
            log basis y=10,
            ymin=0.8,
            ymax=13,
            ytick={1,2,4,6,8,13},
            yticklabels={1,2,4,6,8,13},
            symbolic x coords={c22,c23,c24,c33,c34,c35,c36,c44,c45},
            xtick={c22,c23,c24,c33,c34,c35,c36,c44,c45},
            xticklabels={(2,2),(2,3),(2,4),(3,3),(3,4),(3,5),(3,6),(4,4), (4,5)},
            enlarge x limits=0.05,
            grid=both,
            ylabel={Ratio},
            grid style=dashed,
            major grid style={line width=.2pt,draw=gray!60},
            minor grid style={line width=.1pt,draw=gray!30},
            legend to name=sharedlegend,
            legend columns=4,
            legend style={
                font=\scriptsize,
                /tikz/every even column/.append style={column sep=6pt},
                rounded corners=1pt,
            }
        ]

        \addplot[
            only marks,
            color=black!70,
            mark=triangle*,
            mark size=2.5pt,
        ] coordinates {
            (c22, 2.766817762812924)
            (c23, 2.654965115142849)
            (c24, 4.172335710238858)
            (c33, 4.617818101690229)
            (c34, 5.5272587693481725)
            (c35, 4.603770776815129)
            (c36, nan)
            (c44, 12.357866238631118)
            (c45, nan)
        };
        \addlegendentry{App.~(i): baseline}

        \addplot[
            only marks,
            color=sorRed,
            mark=*,
            mark size=2.3pt,
        ] coordinates {
            (c22, 1.689934337651635)
            (c23, 1.258341686262989)
            (c24, 1.3140542467249818)
            (c33, 1.8749163240355013)
            (c34, 1.8193893797047178)
            (c35, 1.5050575742903145)
            (c36, 1.7161294372294227)
            (c44, 2.710570533263374)
            (c45, 2.3739123823671044)
        };
        \addlegendentry{App.~(ii): \eqref{eq:jensen}, \eqref{eq:dominant_scenario}}

        \addplot[
            only marks,
            color=sorGreen,
            mark=square*,
            mark size=2.3pt,
        ] coordinates {
            (c22, 1.1389843927713657)
            (c23, 1.2122365568582867)
            (c24, 1.198504927807829)
            (c33, 1.5597947556689602)
            (c34, 1.7506897579198577)
            (c35, 1.4859866065620002)
            (c36, 1.7219282339162403)
            (c44, 2.6062984102573448)
            (c45, 2.525733302953374)
        };
        \addlegendentry{App.~(iii): \eqref{eq:dff}}

        \addplot[
            only marks,
            color=ORBlue,
            mark=star,
            mark size=3.5pt,
            mark options={line width=1.2pt},
        ] coordinates {
            (c22,1.0)
            (c23,1.0)
            (c24,1.0)
            (c33,1.0)
            (c34,1.0)
            (c35,1.0)
            (c36,1.0)
            (c44,1.0)
            (c45,1.0)
        };
        \addlegendentry{App.~(iv): \eqref{eq:dff}, \eqref{eq:dffs}, \eqref{eq:dffd}}

        \end{axis}
    \end{tikzpicture}
    \caption{}
    \label{fig:runboth}
\end{subfigure}
\hfill
\begin{subfigure}[b]{0.35\linewidth}
    \centering
    \begin{tikzpicture}
    \begin{axis}[
        width=\linewidth,
        height=5.5cm,
        ybar,
        bar width=6pt,
        ymin=0,
        ymax=0.28,
        symbolic x coords={c46,c47,c48},
        xtick=data,
        xticklabels={(4,6),(4,7),(4,8)},
        enlarge x limits=0.3,
        ytick distance=0.10,
        minor y tick num=3,
        grid=both,
        ylabel={Gap},
        grid style=dashed,
        major grid style={line width=.2pt,draw=gray!70},
        minor grid style={line width=.1pt,draw=gray!50},
    ]

    \addplot[draw=sorRed, fill=sorRed!100] coordinates {
        
        (c46, 0.16772370423223865)
        (c47, 0.22196786249002468)
        (c48, 0.26126353397908714)
    };

    \addplot[draw=sorGreen, fill=sorGreen!100] coordinates {
        
        (c46, 0.15617509400675036)
        (c47, 0.22015385736283552)
        (c48, 0.2653575679805603)
    };

    \addplot[draw=ORBlue, fill=ORBlue!100] coordinates {
        
        (c46, 0.0990773806546518)
        (c47, 0.16243009283130716)
        (c48, 0.21257889370978864)
    };

    \end{axis}
    \end{tikzpicture}
    \caption{}
    \label{fig:gapboth}
\end{subfigure}
\caption{(a) The relative solution time of Approaches~(i)--(iv) under Regime~(1). Values are normalized by Approach~(iv). 
(b) The optimality gap of Approaches~(ii)--(iv) under Regime~(1).
Approach~(i) fails to provide a meaningful upper bound within the time limit.}
\label{fig:both}
\end{figure}

As shown in Figure~\ref{fig:runboth}, Approach~(iv) achieves the lowest solution time across all solvable instances.
Although it entails the largest number of additional inequalities in the~\eqref{eq:rmp}, this overhead is offset by a substantial reduction in the number of iterations required for convergence (see Table~\ref{tab:both}).
The drop in the number of iterations results partly from visiting significantly fewer distributions compared to other approaches (as evidenced by VD values in Table~\ref{tab:both}).
More specifically, Approach~(iv) effectively curtails the exploration of both supply and demand distributions (as shown by RVSD and RVDD metrics in Table~\ref{tab:both}).
Moreover, it has the smallest optimality gap as depicted in Figure~\ref{fig:gapboth}, indicating that early strengthening of the~\eqref{eq:rmp} is valuable when time is limited.

Approaches~(ii) and~(iii) are comparable.
The former strengthens the~\eqref{eq:rmp} with the~\eqref{eq:jensen} and one~\eqref{eq:dominant_scenario} in each iteration of Algorithm~\ref{alg:decomposition}.
If all possible \eqref{eq:dominant_scenario}s are added to the~\eqref{eq:rmp} in Approach~(ii), the resulting model's formulation is effectively as strong as Approach~(iii) at the beginning of the solution process (excluding \eqref{eq:final_benders_cuts} in both cases).
Therefore, their performance differences are primarily driven by the order in which the~\eqref{eq:rmp} is strengthened.
Approach~(iii) starts with a stronger \eqref{eq:rmp} in the beginning, whereas Approach~(ii) has a smaller \eqref{eq:rmp}, which grows with an extra cut per iteration.
As demonstrated in Figure~\ref{fig:both}, both approaches exhibit similar solution times and optimality gaps across all instances.
Finally, the performance difference between Approaches~(iii) and~(iv) shows the significance of imposing distribution-free formulations for subsets of supply and demand distributions, as formulated in \eqref{eq:dffs} and \eqref{eq:dffd}.

We established that our distribution-free formulations in Approach~(iv) significantly enhance the performance of the baseline approach.
Next, we investigate the individual contributions of~\eqref{eq:dffs} and \eqref{eq:dffd}, and whether combining them yields additional benefits.
To this end, we compare the performance of Approaches~(iii)--(vi).
We conduct a sensitivity analysis on the 10 instances of size $(3,6)$ by increasing their mean product demands from $100$ to $180$ units, while keeping demand CV and all other parameters unchanged.
This shift tightens available capacity without altering the underlying uncertainty structure.
Note that instances of the size $(3,6)$ are sufficiently difficult as they are unsolvable by Approach~(i) within the time limit.
Table~\ref{tab:stress_demand} in Appnedix~\ref{sec:appendix_tables} summarizes the performance results for Approaches~(iii)--(vi) under varying demand levels.
Furthermore, Figure~\ref{fig:stress} depicts the solution times of Approaches~(iii)--(vi) relative to Approach~(iii).

First, we focus on the individual improvements due to the~\eqref{eq:dffs} and~\eqref{eq:dffd} when combined with the~\eqref{eq:dff}.
As the demand level increases, the solution time of Approach~(vi) generally improves with respect to Approach~(iii), while the solution time of Approach~(v) deteriorates at higher demand levels.
Therefore, the marginal benefits of the~\eqref{eq:dffs} and~\eqref{eq:dffd} are increasing and decreasing in the demand level, respectively.

These observations can be explained by how the dominant scenarios $\hat{\xi}$, $\hat{\xi}(i',k')$, and $\hat{\xi}(j', Z)$ evolve across varying demand levels.
Recall that the dominant scenarios determine the demand and capacity values on the right-hand side of \eqref{eq:dff}, \eqref{eq:dffs}, and \eqref{eq:dffd} formulations, respectively (except for the first constraint in each formulation).
By our design of $\Gamma$ and the uncertainty mechanism, the $\mu^c$ values are at least $110$ (with total mean capacity bounded between $660$ and $780$), and the efficiency loss is at $1.62\%$ per additional product. 
When the mean demand is low to moderate, given the demand loss of 11\% and 30\% due to mixed and foreign sourcing, the dominant scenarios are likely to set smaller right-hand sides for demand constraints compared to capacity constraints in each formulation.
Note that for~\eqref{eq:dffs}, if capacity constraints are not binding, the supply loss cannot restrict the corresponding production quantities, whereas binding demand constraints allow~\eqref{eq:dffd} to set tighter bounds on production quantities.
In fact, when the demand level is below $120$ units, most performance improvements are attained by adding~\eqref{eq:dff} and~\eqref{eq:dffd}. 
Increasing the demand level relaxes the demand constraints relative to capacity constraints in all formulations, raising the likelihood that the latter becomes binding.

Considering the best-performing combination, Table~\ref{tab:stress_demand} shows that Approach~(iv) consistently outperforms all other combinations.
We conclude that in the absence of structural knowledge on whether supply or demand constraints are binding in the distribution-free cuts' formulations, it is most reliable to employ Approach~(iv).

\begin{figure}
\centering
    \begin{tikzpicture}
        \begin{axis}[
            width=0.8\linewidth,
            height=5.5cm,
            ymin=0.55,
            ytick={0.6, 0.8, 1},
            yticklabels={0.6, 0.8, 1},
            symbolic x coords={c100, c110, c120, c130, c140, c150, c160, c170, c180},
            xtick=data,
            xticklabels={100, 110, 120, 130, 140, 150, 160, 170, 180},
            enlarge x limits=0.05,
            grid=both,
            ylabel={Ratio},
            grid style=dashed,
            major grid style={line width=.2pt,draw=gray!70},
            minor grid style={line width=.1pt,draw=gray!50},
            legend style={
                at={(0.5,1.06)},
                anchor=south,
                legend columns=4,
                font=\scriptsize,
                legend cell align=left,
                /tikz/every even column/.append style={column sep=6pt},
                draw=black,
                fill=white,
                fill opacity=0.95,
                text opacity=1,
                rounded corners=1pt,
            },
        ]

        \addplot[
            only marks,
            color=sorGreen,
            mark=square*,
            mark size=2.1pt,
        ] coordinates {
            (c100,1.0)
            (c110,1.0)
            (c120,1.0)
            (c130,1.0)
            (c140,1.0)
            (c150,1.0)
            (c160,1.0)
            (c170,1.0)
            (c180,1.0)
        };
        \addlegendentry{App.~(iii): \eqref{eq:dff}}

        \addplot[
            only marks,
            color=ORBlue,
            mark=star,
            mark size=3.3pt,
            mark options={line width=1.2pt},
        ] coordinates {
            (c100, 0.580744296)
            (c110, 0.701363602)
            (c120, 0.727394156)
            (c130, 0.676395003)
            (c140, 0.660435763)
            (c150, 0.660679567)
            (c160, 0.67473311)
            (c170, 0.668029142)
            (c180, 0.639775937)
        };
        \addlegendentry{App.~(iv): \eqref{eq:dff}, \eqref{eq:dffs}, \eqref{eq:dffd}}

        \addplot[
            only marks,
            color=red,
            mark=x,
            mark size=3pt,
            mark options={draw=red, fill=white, line width=0.8pt},
            ]  coordinates {
               (c100, 0.598036022)
               (c110, 0.710374426)
               (c120, 0.747583402)
               (c130, 0.730162197)
               (c140, 0.732068485)
               (c150, 0.765752079)
               (c160, 0.86232078)
               (c170, 0.884325247)
               (c180, 0.891203463)
        };
        \addlegendentry{App.~(v): \eqref{eq:dff}, \eqref{eq:dffd}}

        \addplot[
            only marks,
            color=black,
            mark=diamond,
            mark size=2.9pt,
            mark options={draw=black, fill=white, line width=0.8pt},
        ] coordinates {
            (c100, 0.949657032)
            (c110, 0.983716859)
            (c120, 0.946610942)
            (c130, 0.909282456)
            (c140, 0.86658172)
            (c150, 0.810232549)
            (c160, 0.7668503)
            (c170, 0.741544404)
            (c180, 0.722138208)
        };
        \addlegendentry{App.~(vi): \eqref{eq:dff}, \eqref{eq:dffs}}
        \end{axis}
    \end{tikzpicture}
    \caption{The relative solution time of Approaches~(iii)--(vi)  under Regime~(1), when the product demand increases from 100 to 180 units. Values are normalized by Approach~(iii).}
    \label{fig:stress}
\end{figure}

%% file: exo_sup_endo_dem.tex
Next, we assume that only demand uncertainty is endogenous (supply uncertainty is exogenous) and compare Approaches~(i)--(iv).
Table~\ref{tab:dem} in Appendix~\ref{sec:appendix_tables} summarizes the performance results for all instance sizes.
Since supply uncertainty is exogenous, $|\mathcal{P}| = 2^{2 \cdot |J|}$ which is substantially smaller than in Section~\ref{sec:endoboth}.
Therefore, Approach~(i) is able to solve instances up to size $(4,6)$ to optimality.
Analogous to Section~\ref{sec:endoboth}, Approaches~(ii) and~(iii) have comparable performance.
Finally, Approach~(iv) successfully limits the number of visited demand distributions (achieves the lowest RVDD), which in turn significantly reduces the number of iterations.
Therefore, it offers the best solution time and optimality gap.

%% file: endo_sup_exo_dem.tex
Finally, we compare Approaches~(i)--(iv) when supply uncertainty is endogenous (demand uncertainty is exogenous).
Table~\ref{tab:sup} in Appendix~\ref{sec:appendix_tables} summarizes the performance results for all instance sizes.
By our choice of instance sizes, $|\mathcal{P}| = (|J|+1)^{|I|}$ is smaller than the corresponding number of distributions for all but the $(4,4)$ and $(4,5)$ instances in Section~\ref{sec:endodem}, and it is smaller compared to number of distribution for all instances in Section~\ref{sec:endoboth}.
Therefore, Approach~(i) solves all instances up to size $(4,6)$ to optimality and obtains optimality gaps of $10\%$ and $16\%$ for instances of sizes $(4,7)$ and $(4,8)$, respectively.
As observed in sections~\ref{sec:endoboth} and~\ref{sec:endodem}, Approaches~(ii) and~(iii) show similar performance.
However, contrary to the previous observations, Approach~(iv) does not necessarily improve the solution time over Approach~(iii).
While in most instances Approach~(iv) reduces the number of iterations, the computational burden of solving the~\eqref{eq:rmp} with the additional constraints in~\eqref{eq:dffs} outweighs the benefits of fewer iterations.
We offer two explanations for this observation.
First, the smaller size of $\mathcal{P}$ compared to the previous sections decreases the need for additional distribution-free formulations.
Second, the choice of mean demand values for instances at $100$ units per product prevents the capacity constraints in \eqref{eq:dffs} from becoming binding, which in turn limits the effectiveness of these formulations (as shown in Table~\ref{tab:stress_demand} of Appendix~\ref{sec:appendix_tables}).

%% file: managerial.tex
\section{Managerial Insights: An Illustrative Example}\label{sec:managerial}
Choosing appropriate parameter values for the endogenous uncertainty mechanism can be challenging.
The impact of a plant's degree on the supply distribution can be quantified relatively directly because of its operational nature.
In contrast, modeling the demand distribution based on the zone assignments of a product is less direct.
For a product only sourced from a single zone, a market analysis can assess the country-of-origin and tariff effects.
When a product is sourced from multiple zones (mixed sourcing), the firm's ability to allocate production volumes across plants and zones may further affect the induced demand distribution.
This motivates a sensitivity analysis with respect to the parameters governing the endogenous effects.
It further provides insight into the robustness of the design with respect to parameter perturbations.

We perform a sensitivity analysis by examining different levels of mixed-sourcing mean demand loss and supply efficiency loss under two supply-variability settings: a low- and a high-variability setting.
This allows us to examine not only how endogenous effects change the optimal design, but also whether the direction and magnitude of this change depend on the underlying level of uncertainty.
Furthermore, we highlight the value of modeling endogenous effects by comparing the optimal design with the design that is optimal in the absence of endogenous effects, i.e., the exogenous-optimal design.
Under low supply variability, the exogenous-optimal design has minimal flexibility (each product is sourced from a single plant), whereas under high supply variability, the exogenous-optimal design is highly flexible (each pair of plants shares one product).
We demonstrate that the same endogenous effects may lead to different deviations from the exogenous-optimal design across the two settings.

\subsection{Example Setup}

Let $|I| = 3$ and $|J| = 6$, and assume that all plants and all products are homogeneous (allowing us to isolate endogenous effects).
We set $r_{i,j} = 1$, $s_{i,j} = 100$, and $q_{i,j} = 7$ for all $i \in I$ and $j \in J$.
Product demands are independent and normally distributed with $\mu_j^d = 100$ and $\sigma_j^d = 40$ for all $j \in J$. 
Supply capacities are also independent and normally distributed, with $\mu_i^c = 220$.
For low- and high-variability supply settings, we set $\sigma_i^c = 22$ and $\sigma_i^c = 88$ for all $i \in I$, respectively.
The choice of mechanism for $\mathbb{P}_{d,h}$ is the same as in Section~\ref{sec:setup_expo}:
on the supply side, if plant $i$ is assigned to $d_{i}$ products, then $\mu^c_{i,d_i} = \mu^c_{i} \cdot (1 - \lambda \cdot (d_{i} - 1))$, where $\lambda$ denotes the efficiency loss per product at a given plant, and $\sigma^c_{i,d_i} = \nicefrac{\sigma_i^c}{\mu_i^c} \cdot \mu^c_{i,d_i}$.
On the demand side, we assume that Plants~$1$ and~$2$ are in Zone~$0$ (domestic) and Plant~$3$ is in Zone~$1$ (foreign).
For product $j$, if $h(j) = \{0\}$, then $\mu_{j, h(j)}^d = \mu_j^d$ and $\sigma_{j, h(j)}^d = \sigma_j^d$.
If $h(j) = \{1\}$, then $\mu_{j, h(j)}^d = \mu_j^d \cdot (1 - 0.3)$ and $\sigma_{j, h(j)}^d = (1+\nicefrac{0.3}{3}) \cdot \nicefrac{\sigma_j^d}{\mu_j^d} \cdot \mu_{j, h(j)}^d$.
Finally, when $h(j) = \{0,1\}$, we set $\mu_{j, h(j)}^d = \mu_j^d \cdot (1 - \theta)$ and $\sigma_{j, h(j)}^d = (1+\nicefrac{\theta}{3}) \cdot \nicefrac{\sigma_j^d}{\mu_j^d} \cdot \mu_{j, h(j)}^d$, where $\theta$ denotes the demand loss in the case of mixed sourcing.
In order to compute the exogenous-optimal design, we remove the endogenous effects, i.e., the parameters for the supply distributions are $\mu_i^c$ and $\sigma_i^c$ at plant $i$, and the parameters for the demand distributions are $\mu_j^d$ and $\sigma_j^d$ for product $j$.

We perform our analysis under both supply variability settings for $\theta$ in the range of 7\% to 15\% (larger $\theta$ implies smaller demand benefits from pooling domestic and foreign plants) and for $\lambda$ in the range of $1\%$ to $6\%$ (larger $\lambda$ values pertain to industries with more significant changeover times, e.g., the process industry).
For each combination of $\theta$ and $\lambda$, we report the optimal flexibility design and the objective function gap relative to the exogenous-optimal design to illustrate the importance of capturing the endogenous effects (see Table~\ref{tab:managerial_combined}). 
Figure~\ref{fig:illustrative_example} depicts the optimal designs from Table~\ref{tab:managerial_combined}.
All instances are solved to optimality using Approach~(iv).
At each iteration of Algorithm~\ref{alg:decomposition}, we use $4{,}000$ scenarios to approximate the second-stage profit.
Because different designs induce different distributions, the performance of the optimal design and the exogenous-optimal design are evaluated under their respective distributions.
We evaluate each design out-of-sample using 20 independent Monte Carlo replications of ${20{,}000}$ scenarios and compute the average profit across all replications.
The objective function gap (OFG) of the optimal design relative to the exogenous-optimal design is then calculated as
$\text{OFG} = \nicefrac{(Z^* - Z^{\text{ex.-opt.}})}{Z^{\text{ex.-opt.}}},$
where $Z^*$ and $Z^{\text{ex.-opt.}}$ denote the average objective values of the optimal design and the exogenous-optimal design, respectively.
To distinguish between flexibility designs, we refer to them by their plants' node degrees using the tuple $(d_1, d_2, d_3)$.

\subsection{Results}

In the low-variability setting, the exogenous-optimal design is $(2,2,2)$ (see Figure~\ref{fig:illustrative_example}).
When either mixed sourcing substantially decreases market demand ($\theta$ is high) or the supply loss due to additional flexibility is high ($\lambda$ is high), the optimal design remains $(2,2,2)$.
That is, adding domestic links to products otherwise served from the foreign zone (reshoring products 5 and 6) is not favorable, because the demand gain from mixed sourcing relative to exclusive foreign sourcing is insufficient to compensate for the substantial flexibility investment costs and the additional efficiency loss due to increased product variety.

Conversely, as moving from foreign to mixed sourcing becomes more favorable ($\theta$ decreases) or efficiency loss becomes less severe ($\lambda$ decreases), the optimal design shifts from $(2,2,2)$ to $(2,3,2)$ (with negligible objective differences) or to $(3,3,2)$.
At $\theta = 9\%$ and $\lambda = 5\%$, the benefits of partially reshoring Product~5 (in the form of improved market demand) are offset by the additional efficiency loss incurred by assigning it to Plant~2 and the associated investment cost.
Hence, the OFG is zero.
When moving from foreign to mixed sourcing significantly improves the market demand and efficiency loss is relatively small, the optimal design changes to $(3,3,2)$, partially reshoring both Products~5 and~6 through Plants~2 and~1, respectively.
The OFG generally increases as $\lambda$ and $\theta$ decrease.

\begin{figure}
    \centering
    \footnotesize

    \tikzset{every picture/.style={line width=0.75pt}}      

    \begin{tikzpicture}[x=0.55pt,y=0.55pt,yscale=-0.9,xscale=0.9]

    \draw  [color={rgb, 255:red, 65; green, 117; blue, 5 }  ,draw opacity=0.38 ][fill={rgb, 255:red, 65; green, 117; blue, 5 }  ,fill opacity=0.11 ] (375.07,4643.1) .. controls (375.07,4632.88) and (383.35,4624.6) .. (393.57,4624.6) -- (393.57,4624.6) .. controls (403.78,4624.6) and (412.07,4632.88) .. (412.07,4643.1) -- (412.07,4730.1) .. controls (412.07,4740.32) and (403.78,4748.6) .. (393.57,4748.6) -- (393.57,4748.6) .. controls (383.35,4748.6) and (375.07,4740.32) .. (375.07,4730.1) -- cycle ;
    \draw  [color={rgb, 255:red, 65; green, 117; blue, 5 }  ,draw opacity=0.38 ][fill={rgb, 255:red, 65; green, 117; blue, 5 }  ,fill opacity=0.11 ] (209.07,4643.1) .. controls (209.07,4632.88) and (217.35,4624.6) .. (227.57,4624.6) -- (227.57,4624.6) .. controls (237.78,4624.6) and (246.07,4632.88) .. (246.07,4643.1) -- (246.07,4730.1) .. controls (246.07,4740.32) and (237.78,4748.6) .. (227.57,4748.6) -- (227.57,4748.6) .. controls (217.35,4748.6) and (209.07,4740.32) .. (209.07,4730.1) -- cycle ;
    \draw  [color={rgb, 255:red, 65; green, 117; blue, 5 }  ,draw opacity=0.38 ][fill={rgb, 255:red, 65; green, 117; blue, 5 }  ,fill opacity=0.11 ] (47.07,4644.1) .. controls (47.07,4633.88) and (55.35,4625.6) .. (65.57,4625.6) -- (65.57,4625.6) .. controls (75.78,4625.6) and (84.07,4633.88) .. (84.07,4644.1) -- (84.07,4731.1) .. controls (84.07,4741.32) and (75.78,4749.6) .. (65.57,4749.6) -- (65.57,4749.6) .. controls (55.35,4749.6) and (47.07,4741.32) .. (47.07,4731.1) -- cycle ;
    \draw  [dash pattern={on 4.5pt off 4.5pt}]  (359.07,4585.17) -- (359.07,4872) ;
    \draw  [dash pattern={on 4.5pt off 4.5pt}]  (192.07,4585.17) -- (192.07,4872) ;
    \draw  [color={rgb, 255:red, 65; green, 117; blue, 5 }  ,draw opacity=0.38 ][fill={rgb, 255:red, 65; green, 117; blue, 5 }  ,fill opacity=0.11 ] (539,4643.1) .. controls (539,4632.88) and (547.28,4624.6) .. (557.5,4624.6) -- (557.5,4624.6) .. controls (567.72,4624.6) and (576,4632.88) .. (576,4643.1) -- (576,4730.1) .. controls (576,4740.32) and (567.72,4748.6) .. (557.5,4748.6) -- (557.5,4748.6) .. controls (547.28,4748.6) and (539,4740.32) .. (539,4730.1) -- cycle ;
    \draw  [dash pattern={on 4.5pt off 4.5pt}]  (525,4585.17) -- (525,4872) ;

    \draw    (324.47, 4699.6) circle [x radius= 13.73, y radius= 13.73]   ;
    \draw (324.47,4699.6) node   [align=left] {3};
    \draw  [color={rgb, 255:red, 74; green, 74; blue, 74 }  ,draw opacity=1 ]  (226.44, 4649.6) circle [x radius= 13.73, y radius= 13.73]   ;
    \draw (226.44,4649.6) node   [align=left] {\textcolor[rgb]{0.29,0.29,0.29}{1}};
    \draw    (324.44, 4849.6) circle [x radius= 13.73, y radius= 13.73]   ;
    \draw (324.44,4849.6) node   [align=left] {6};
    \draw    (324.44, 4799.6) circle [x radius= 13.73, y radius= 13.73]   ;
    \draw (324.44,4799.6) node   [align=left] {5};
    \draw    (324.44, 4749.6) circle [x radius= 13.73, y radius= 13.73]   ;
    \draw (324.44,4749.6) node   [align=left] {4};
    \draw    (324.44, 4649.6) circle [x radius= 13.73, y radius= 13.73]   ;
    \draw (324.44,4649.6) node   [align=left] {2};
    \draw    (324.44, 4599.6) circle [x radius= 13.73, y radius= 13.73]   ;
    \draw (324.44,4599.6) node   [align=left] {1};
    \draw  [color={rgb, 255:red, 74; green, 74; blue, 74 }  ,draw opacity=1 ]  (226.44, 4724.6) circle [x radius= 13.73, y radius= 13.73]   ;
    \draw (226.44,4724.6) node   [align=left] {\textcolor[rgb]{0.29,0.29,0.29}{2}};
    \draw  [color={rgb, 255:red, 74; green, 74; blue, 74 }  ,draw opacity=1 ]  (226.44, 4799.6) circle [x radius= 13.73, y radius= 13.73]   ;
    \draw (226.44,4799.6) node   [align=left] {\textcolor[rgb]{0.29,0.29,0.29}{3}};
    \draw    (491.47, 4699.6) circle [x radius= 13.73, y radius= 13.73]   ;
    \draw (491.47,4699.6) node   [align=left] {3};
    \draw  [color={rgb, 255:red, 74; green, 74; blue, 74 }  ,draw opacity=1 ]  (393.44, 4649.6) circle [x radius= 13.73, y radius= 13.73]   ;
    \draw (393.44,4649.6) node   [align=left] {\textcolor[rgb]{0.29,0.29,0.29}{1}};
    \draw    (491.44, 4849.6) circle [x radius= 13.73, y radius= 13.73]   ;
    \draw (491.44,4849.6) node   [align=left] {6};
    \draw    (491.44, 4799.6) circle [x radius= 13.73, y radius= 13.73]   ;
    \draw (491.44,4799.6) node   [align=left] {5};
    \draw    (491.44, 4749.6) circle [x radius= 13.73, y radius= 13.73]   ;
    \draw (491.44,4749.6) node   [align=left] {4};
    \draw    (491.44, 4649.6) circle [x radius= 13.73, y radius= 13.73]   ;
    \draw (491.44,4649.6) node   [align=left] {2};
    \draw    (491.44, 4599.6) circle [x radius= 13.73, y radius= 13.73]   ;
    \draw (491.44,4599.6) node   [align=left] {1};
    \draw  [color={rgb, 255:red, 74; green, 74; blue, 74 }  ,draw opacity=1 ]  (393.44, 4724.6) circle [x radius= 13.73, y radius= 13.73]   ;
    \draw (393.44,4724.6) node   [align=left] {\textcolor[rgb]{0.29,0.29,0.29}{2}};
    \draw  [color={rgb, 255:red, 74; green, 74; blue, 74 }  ,draw opacity=1 ]  (393.44, 4799.6) circle [x radius= 13.73, y radius= 13.73]   ;
    \draw (393.44,4799.6) node   [align=left] {\textcolor[rgb]{0.29,0.29,0.29}{3}};
    \draw    (163.47, 4699.6) circle [x radius= 13.73, y radius= 13.73]   ;
    \draw (163.47,4699.6) node   [align=left] {3};
    \draw  [color={rgb, 255:red, 74; green, 74; blue, 74 }  ,draw opacity=1 ]  (65.44, 4649.6) circle [x radius= 13.73, y radius= 13.73]   ;
    \draw (65.44,4649.6) node   [align=left] {\textcolor[rgb]{0.29,0.29,0.29}{1}};
    \draw    (163.44, 4849.6) circle [x radius= 13.73, y radius= 13.73]   ;
    \draw (163.44,4849.6) node   [align=left] {6};
    \draw    (163.44, 4799.6) circle [x radius= 13.73, y radius= 13.73]   ;
    \draw (163.44,4799.6) node   [align=left] {5};
    \draw    (163.44, 4749.6) circle [x radius= 13.73, y radius= 13.73]   ;
    \draw (163.44,4749.6) node   [align=left] {4};
    \draw    (163.44, 4649.6) circle [x radius= 13.73, y radius= 13.73]   ;
    \draw (163.44,4649.6) node   [align=left] {\textcolor[rgb]{0.29,0.29,0.29}{2}};
    \draw    (163.44, 4599.6) circle [x radius= 13.73, y radius= 13.73]   ;
    \draw (163.44,4599.6) node   [align=left] {1};
    \draw  [color={rgb, 255:red, 74; green, 74; blue, 74 }  ,draw opacity=1 ]  (65.44, 4724.6) circle [x radius= 13.73, y radius= 13.73]   ;
    \draw (65.44,4724.6) node   [align=left] {\textcolor[rgb]{0.29,0.29,0.29}{2}};
    \draw  [color={rgb, 255:red, 74; green, 74; blue, 74 }  ,draw opacity=1 ]  (65.44, 4799.6) circle [x radius= 13.73, y radius= 13.73]   ;
    \draw (65.44,4799.6) node   [align=left] {\textcolor[rgb]{0.29,0.29,0.29}{3}};
    \draw (78.23,4847.83) node [anchor=north west][inner sep=0.75pt]  [font=\small,color={rgb, 255:red, 74; green, 74; blue, 74 }  ,opacity=1 ] [align=left] {$\displaystyle ( 2,2,2)$};
    \draw (240.4,4848.8) node [anchor=north west][inner sep=0.75pt]  [font=\small,color={rgb, 255:red, 74; green, 74; blue, 74 }  ,opacity=1 ] [align=left] {$\displaystyle ( 2,3,2)$};
    \draw (407,4847.83) node [anchor=north west][inner sep=0.75pt]  [font=\small,color={rgb, 255:red, 74; green, 74; blue, 74 }  ,opacity=1 ] [align=left] {$\displaystyle ( 3,3,2)$};
    \draw (289.4,4561.6) node [anchor=north west][inner sep=0.75pt]  [font=\small,color={rgb, 255:red, 74; green, 74; blue, 74 }  ,opacity=1 ] [align=left] {$\displaystyle \text{Products}$};
    \draw (203.73,4562.27) node [anchor=north west][inner sep=0.75pt]  [font=\small,color={rgb, 255:red, 74; green, 74; blue, 74 }  ,opacity=1 ] [align=left] {$\displaystyle \text{Plants}$};
    \draw (123.4,4561.6) node [anchor=north west][inner sep=0.75pt]  [font=\small,color={rgb, 255:red, 74; green, 74; blue, 74 }  ,opacity=1 ] [align=left] {$\displaystyle \text{Products}$};
    \draw (35.73,4562.27) node [anchor=north west][inner sep=0.75pt]  [font=\small,color={rgb, 255:red, 74; green, 74; blue, 74 }  ,opacity=1 ] [align=left] {$\displaystyle \text{Plants}$};
    \draw (454,4561.6) node [anchor=north west][inner sep=0.75pt]  [font=\small,color={rgb, 255:red, 74; green, 74; blue, 74 }  ,opacity=1 ] [align=left] {$\displaystyle \text{Products}$};
    \draw (377.33,4562.27) node [anchor=north west][inner sep=0.75pt]  [font=\small,color={rgb, 255:red, 74; green, 74; blue, 74 }  ,opacity=1 ] [align=left] {$\displaystyle \text{Plants}$};
    \draw    (655.4, 4699.6) circle [x radius= 13.73, y radius= 13.73]   ;
    \draw (655.4,4699.6) node   [align=left] {3};
    \draw  [color={rgb, 255:red, 74; green, 74; blue, 74 }  ,draw opacity=1 ]  (557.37, 4649.6) circle [x radius= 13.73, y radius= 13.73]   ;
    \draw (557.37,4649.6) node   [align=left] {\textcolor[rgb]{0.29,0.29,0.29}{1}};
    \draw    (655.38, 4849.6) circle [x radius= 13.73, y radius= 13.73]   ;
    \draw (655.38,4849.6) node   [align=left] {6};
    \draw    (655.38, 4799.6) circle [x radius= 13.73, y radius= 13.73]   ;
    \draw (655.38,4799.6) node   [align=left] {5};
    \draw    (655.38, 4749.6) circle [x radius= 13.73, y radius= 13.73]   ;
    \draw (655.38,4749.6) node   [align=left] {4};
    \draw    (655.38, 4649.6) circle [x radius= 13.73, y radius= 13.73]   ;
    \draw (655.38,4649.6) node   [align=left] {2};
    \draw    (655.38, 4599.6) circle [x radius= 13.73, y radius= 13.73]   ;
    \draw (655.38,4599.6) node   [align=left] {1};
    \draw  [color={rgb, 255:red, 74; green, 74; blue, 74 }  ,draw opacity=1 ]  (557.38, 4724.6) circle [x radius= 13.73, y radius= 13.73]   ;
    \draw (557.38,4724.6) node   [align=left] {\textcolor[rgb]{0.29,0.29,0.29}{2}};
    \draw  [color={rgb, 255:red, 74; green, 74; blue, 74 }  ,draw opacity=1 ]  (557.38, 4799.6) circle [x radius= 13.73, y radius= 13.73]   ;
    \draw (557.38,4799.6) node   [align=left] {\textcolor[rgb]{0.29,0.29,0.29}{3}};
    \draw (570.94,4847.83) node [anchor=north west][inner sep=0.75pt]  [font=\small,color={rgb, 255:red, 74; green, 74; blue, 74 }  ,opacity=1 ] [align=left] {$\displaystyle ( 3,3,3)$};
    \draw (628.93,4561.6) node [anchor=north west][inner sep=0.75pt]  [font=\small,color={rgb, 255:red, 74; green, 74; blue, 74 }  ,opacity=1 ] [align=left] {$\displaystyle \text{Products}$};
    \draw (541.27,4562.27) node [anchor=north west][inner sep=0.75pt]  [font=\small,color={rgb, 255:red, 74; green, 74; blue, 74 }  ,opacity=1 ] [align=left] {$\displaystyle \text{Plants}$};
    \draw [color={rgb, 255:red, 47; green, 70; blue, 140 }  ,draw opacity=1 ]   (238.67,4643.36) -- (312.21,4605.84) ;
    \draw [color={rgb, 255:red, 47; green, 70; blue, 140 }  ,draw opacity=1 ]   (240.17,4649.6) -- (310.71,4649.6) ;
    \draw [color={rgb, 255:red, 47; green, 70; blue, 140 }  ,draw opacity=1 ]   (239.75,4721.21) -- (311.16,4703) ;
    \draw [color={rgb, 255:red, 47; green, 70; blue, 140 }  ,draw opacity=1 ]   (239.75,4728) -- (311.13,4746.21) ;
    \draw [color={rgb, 255:red, 47; green, 70; blue, 140 }  ,draw opacity=1 ]   (238.67,4805.84) -- (312.21,4843.36) ;
    \draw [color={rgb, 255:red, 47; green, 70; blue, 140 }  ,draw opacity=1 ]   (405.67,4643.36) -- (479.21,4605.84) ;
    \draw [color={rgb, 255:red, 47; green, 70; blue, 140 }  ,draw opacity=1 ]   (407.17,4649.6) -- (477.71,4649.6) ;
    \draw [color={rgb, 255:red, 47; green, 70; blue, 140 }  ,draw opacity=1 ]   (406.75,4728) -- (478.13,4746.21) ;
    \draw [color={rgb, 255:red, 47; green, 70; blue, 140 }  ,draw opacity=1 ]   (404.35,4732.95) -- (480.54,4791.26) ;
    \draw [color={rgb, 255:red, 47; green, 70; blue, 140 }  ,draw opacity=1 ]   (405.67,4805.84) -- (479.21,4843.36) ;
    \draw [color={rgb, 255:red, 47; green, 70; blue, 140 }  ,draw opacity=1 ]   (77.67,4643.36) -- (151.21,4605.84) ;
    \draw [color={rgb, 255:red, 47; green, 70; blue, 140 }  ,draw opacity=1 ]   (79.17,4649.6) -- (149.71,4649.6) ;
    \draw [color={rgb, 255:red, 47; green, 70; blue, 140 }  ,draw opacity=1 ]   (78.75,4721.21) -- (150.16,4703) ;
    \draw [color={rgb, 255:red, 47; green, 70; blue, 140 }  ,draw opacity=1 ]   (78.75,4728) -- (150.13,4746.21) ;
    \draw [color={rgb, 255:red, 47; green, 70; blue, 140 }  ,draw opacity=1 ]   (77.67,4805.84) -- (151.21,4843.36) ;
    \draw [color={rgb, 255:red, 47; green, 70; blue, 140 }  ,draw opacity=1 ]   (79.17,4799.6) -- (149.71,4799.6) ;
    \draw [color={rgb, 255:red, 47; green, 70; blue, 140 }  ,draw opacity=1 ]   (406.75,4721.21) -- (478.16,4703) ;
    \draw [color={rgb, 255:red, 47; green, 70; blue, 140 }  ,draw opacity=1 ]   (399.48,4661.93) -- (485.4,4837.27) ;
    \draw [color={rgb, 255:red, 47; green, 70; blue, 140 }  ,draw opacity=1 ]   (407.17,4799.6) -- (477.71,4799.6) ;
    \draw [color={rgb, 255:red, 47; green, 70; blue, 140 }  ,draw opacity=1 ]   (240.17,4799.6) -- (310.71,4799.6) ;
    \draw [color={rgb, 255:red, 47; green, 70; blue, 140 }  ,draw opacity=1 ]   (237.35,4732.95) -- (313.54,4791.26) ;
    \draw [color={rgb, 255:red, 47; green, 70; blue, 140 }  ,draw opacity=1 ]   (569.6,4643.36) -- (643.14,4605.84) ;
    \draw [color={rgb, 255:red, 47; green, 70; blue, 140 }  ,draw opacity=1 ]   (571.1,4649.6) -- (641.65,4649.6) ;
    \draw [color={rgb, 255:red, 47; green, 70; blue, 140 }  ,draw opacity=1 ]   (570.68,4728) -- (642.07,4746.21) ;
    \draw [color={rgb, 255:red, 47; green, 70; blue, 140 }  ,draw opacity=1 ]   (568.28,4732.95) -- (644.47,4791.26) ;
    \draw [color={rgb, 255:red, 47; green, 70; blue, 140 }  ,draw opacity=1 ]   (569.61,4805.84) -- (643.14,4843.36) ;
    \draw [color={rgb, 255:red, 47; green, 70; blue, 140 }  ,draw opacity=1 ]   (570.68,4721.21) -- (642.09,4703) ;
    \draw [color={rgb, 255:red, 47; green, 70; blue, 140 }  ,draw opacity=1 ]   (571.1,4799.6) -- (641.65,4799.6) ;
    \draw [color={rgb, 255:red, 47; green, 70; blue, 140 }  ,draw opacity=1 ]   (569.6,4655.84) -- (643.17,4693.36) ;
    \draw [color={rgb, 255:red, 47; green, 70; blue, 140 }  ,draw opacity=1 ]   (564.89,4788.11) -- (647.86,4661.1) ;

    \end{tikzpicture}

    \caption{Visual representation of the optimal flexibility designs from Table~\ref{tab:managerial_combined}.
    Tuples represent plants' node degrees.
    Green shaded plants are domestic.}%
    \label{fig:illustrative_example}%
\end{figure}
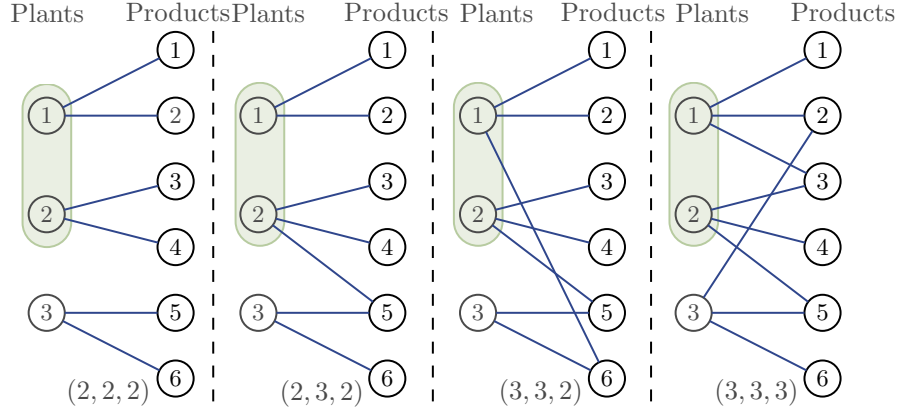

The exogenous-optimal design for the high-variability setting is highly flexible and has plant degrees $(3,3,3)$ (see Figure~\ref{fig:illustrative_example}) to hedge against the elevated supply risk.
However, once endogenous effects are introduced, the exogenous-optimal design becomes costly:
Plant~$3$ loses efficiency when serving three products.
Furthermore, products served by Plant~3 suffer from demand loss either in foreign or mixed-sourcing cases.
Therefore, even at small values for $\lambda$ and $\theta$, the model retreats to Design~$(3,3,2)$ and then to the Design~$(2,2,2)$ as $\lambda$ and $\theta$ increase.
Note that in this setting, Design~$(3,3,3)$ is never optimal over the tested parameter range.

In contrast to the low supply-variability setting, we observe that the OFG behavior is not monotone.
When the optimal design is $(3,3,2)$, the OFG relative to Design~$(3,3,3)$ tends to decrease as $\theta$ or $\lambda$ increases, whereas it increases once the optimal design changes to $(2,2,2)$.
Designs~$(3,3,2)$ and~$(3,3,3)$ both contain two products subject to mixed sourcing.
However, they have competing advantages:
Design~$(3,3,2)$ suffers from a smaller efficiency loss on Plant~$3$ and subjects one less product to foreign sourcing (Product~$6$), whereas Design~$(3,3,3)$ provides more flexibility and therefore more options to match the realized supply and demand under high uncertainty.
As $\theta$ increases, mixed-sourced products face higher demand variability.
Similarly, as $\lambda$ increases, high-degree plants lose more effective capacity.
Both effects deteriorate the performance of the mentioned designs, but they also make flexibility valuable: additional links help hedge against higher demand variability and allocate realized capacity across products more effectively.
Thus, as long as the optimal design remains $(3,3,2)$, the additional flexibility of Design~$(3,3,3)$ remains valuable and the relative advantage of Design~$(3,3,2)$ may decrease.
After the optimal design shifts to $(2,2,2)$ at higher values of $\lambda$ and $\theta$, the model limits flexibility to avoid high efficiency losses and small mixed-sourcing gains (both of which heavily penalize Design~$(3,3,3)$).
Hence, the OFG tends to increase.

Comparing the low- and high-variability settings, at any fixed level of endogenous effects ($\lambda$ and $\theta$) we see that the optimal design under the high-variability setting is at least as flexible as (often more flexible than) the low-variability setting.
For example, at $\theta = 15\%$ and $\lambda = 1\%$, the optimal design under the high-variability setting is $(3,3,2)$ while the optimal design under the low-variability setting is $(2,2,2)$.
In fact, we see that the thresholds at which the optimal design changes from $(3,3,2)$ to $(2,2,2)$ shift toward larger values of $\theta$ and $\lambda$ under high supply variability.

We conclude that the interaction between the problem's uncertainty level and the endogenous effects determines the optimal design in our example and neither effect should be ignored when modeling the problem. 
In addition, the optimal flexibility design under endogenous effects may be either more or less flexible than the exogenous-optimal design, depending on the supply-variability setting and the magnitude of endogenous effects.
Table~\ref{tab:managerial_combined} further illustrates thresholds of $\theta$ and $\lambda$ at which the optimal design shifts, providing insight into the robustness of the optimal design with respect to perturbations in the parameters governing the endogenous effects.

\begin{table}[t]
\centering
\scriptsize
\setlength{\tabcolsep}{3pt}
\renewcommand{\arraystretch}{1.15}
\newcommand{\cellres}[2]{\shortstack{#1\\[-1pt]{\scriptsize #2}}}
\caption{Optimal flexibility designs and objective function gaps relative to the exogenous-optimal designs under low and high supply variability settings. A dash indicates that the optimal design has the same node degrees as the exogenous-optimal design.}
\label{tab:managerial_combined}
\begin{tabular}{rcccccc@{\hspace{10pt}}cccccc}
\toprule
&
\multicolumn{6}{c}{Low supply variability}
&
\multicolumn{6}{c}{High supply variability} \\
\cmidrule(lr){2-7}
\cmidrule(lr){8-13}
& \multicolumn{6}{c}{$\lambda$}
& \multicolumn{6}{c}{$\lambda$} \\
\cmidrule(lr){2-7}
\cmidrule(lr){8-13}
$\theta$ & 1\% & 2\% & 3\% & 4\% & 5\% & 6\%
& 1\% & 2\% & 3\% & 4\% & 5\% & 6\% \\
\midrule

7\%
&
\cIII{\cellres{$(3,3,2)$}{$3.20\%$}}
&
\cIII{\cellres{$(3,3,2)$}{$2.67\%$}}
&
\cIII{\cellres{$(3,3,2)$}{$2.06\%$}}
&
\cIII{\cellres{$(3,3,2)$}{$1.40\%$}}
&
\cIII{\cellres{$(3,3,2)$}{$0.67\%$}}
&
\cII{\cellres{$(2,3,2)$}{$0.02\%$}}
&
\cIII{\cellres{$(3,3,2)$}{$1.61\%$}}
&
\cIII{\cellres{$(3,3,2)$}{$1.51\%$}}
&
\cIII{\cellres{$(3,3,2)$}{$1.43\%$}}
&
\cIII{\cellres{$(3,3,2)$}{$1.39\%$}}
&
\cIII{\cellres{$(3,3,2)$}{$1.38\%$}}
&
\cIII{\cellres{$(3,3,2)$}{$1.41\%$}}
\\[4pt]

9\%
&
\cIII{\cellres{$(3,3,2)$}{$2.38\%$}}
&
\cIII{\cellres{$(3,3,2)$}{$1.85\%$}}
&
\cIII{\cellres{$(3,3,2)$}{$1.27\%$}}
&
\cIII{\cellres{$(3,3,2)$}{$0.62\%$}}
&
\cII{\cellres{$(2,3,2)$}{$0.00\%$}}
&
\cI{\cellres{$(2,2,2)$}{--}}
&
\cIII{\cellres{$(3,3,2)$}{$1.53\%$}}
&
\cIII{\cellres{$(3,3,2)$}{$1.42\%$}}
&
\cIII{\cellres{$(3,3,2)$}{$1.33\%$}}
&
\cIII{\cellres{$(3,3,2)$}{$1.27\%$}}
&
\cIII{\cellres{$(3,3,2)$}{$1.24\%$}}
&
\cIII{\cellres{$(3,3,2)$}{$1.25\%$}}
\\[4pt]

11\%
&
\cIII{\cellres{$(3,3,2)$}{$1.54\%$}}
&
\cIII{\cellres{$(3,3,2)$}{$1.03\%$}}
&
\cIII{\cellres{$(3,3,2)$}{$0.45\%$}}
&
\cI{\cellres{$(2,2,2)$}{--}}
&
\cI{\cellres{$(2,2,2)$}{--}}
&
\cI{\cellres{$(2,2,2)$}{--}}
&
\cIII{\cellres{$(3,3,2)$}{$1.45\%$}}
&
\cIII{\cellres{$(3,3,2)$}{$1.33\%$}}
&
\cIII{\cellres{$(3,3,2)$}{$1.23\%$}}
&
\cIII{\cellres{$(3,3,2)$}{$1.15\%$}}
&
\cIII{\cellres{$(3,3,2)$}{$1.10\%$}}
&
\cI{\cellres{$(2,2,2)$}{$1.43\%$}}
\\[4pt]

13\%
&
\cIII{\cellres{$(3,3,2)$}{$0.68\%$}}
&
\cIII{\cellres{$(3,3,2)$}{$0.18\%$}}
&
\cI{\cellres{$(2,2,2)$}{--}}
&
\cI{\cellres{$(2,2,2)$}{--}}
&
\cI{\cellres{$(2,2,2)$}{--}}
&
\cI{\cellres{$(2,2,2)$}{--}}
&
\cIII{\cellres{$(3,3,2)$}{$1.37\%$}}
&
\cIII{\cellres{$(3,3,2)$}{$1.24\%$}}
&
\cIII{\cellres{$(3,3,2)$}{$1.12\%$}}
&
\cIII{\cellres{$(3,3,2)$}{$1.03\%$}}
&
\cI{\cellres{$(2,2,2)$}{$1.46\%$}}
&
\cI{\cellres{$(2,2,2)$}{$2.11\%$}}
\\[4pt]

15\%
&
\cI{\cellres{$(2,2,2)$}{--}}
&
\cI{\cellres{$(2,2,2)$}{--}}
&
\cI{\cellres{$(2,2,2)$}{--}}
&
\cI{\cellres{$(2,2,2)$}{--}}
&
\cI{\cellres{$(2,2,2)$}{--}}
&
\cI{\cellres{$(2,2,2)$}{--}}
&
\cIII{\cellres{$(3,3,2)$}{$1.29\%$}}
&
\cIII{\cellres{$(3,3,2)$}{$1.14\%$}}
&
\cI{\cellres{$(2,2,2)$}{$1.18\%$}}
&
\cI{\cellres{$(2,2,2)$}{$1.64\%$}}
&
\cI{\cellres{$(2,2,2)$}{$2.18\%$}}
&
\cI{\cellres{$(2,2,2)$}{$2.82\%$}}
\\[4pt]

\bottomrule
\end{tabular}
\end{table}

%% file: conclusion.tex
\section{Conclusion}\label{sec:conclusion}
This paper studied the supply network flexibility design problem under the assumption that the network design affects the supply and demand uncertainties.
We modeled the problem as a two-stage stochastic program with endogenous uncertainty, wherein the second-stage supply and demand uncertainties are functions of the first-stage design.
Specifically, we modeled supply uncertainty as a function of the number of products assigned to a plant to capture the efficiency loss (due to changeovers and product variety) and the specialization effect.
Furthermore, we modeled the impact of the sourcing locations of a product on its demand uncertainty, to reflect the country-of-origin and tariff effects.

Notably, our modeling framework does not rely on the bipartite structure of the network design problem.
Moreover, it does not hinge on a particular functional form or parameterization of the endogenous supply and demand mechanisms.
These features suggest that the proposed modeling framework is applicable to a broader class of network design problems with endogenous uncertainty, wherein the supply or demand distribution at a non-transshipment node depends on its degree or zone assignment.

To solve the problem efficiently, we implemented the L-shaped method and derived distribution-specific optimality cuts as a baseline approach.
Additionally, we proposed new distribution-free formulations that impose valid upper bounds on the expected second-stage profit for all distributions and subsets of supply and demand distributions.
Extensive experiments confirmed that our distribution-free formulations substantially improve the performance of the baseline approach by reducing the number of visited distributions.
We further explained how their performance depends on the problem parameters.
The idea behind our distribution-free formulations, i.e., embedding the second-stage formulation in the relaxed master problem using a dominant scenario, is not confined to our problem setting and remains viable as long as the second-stage formulation admits a small formulation and the dominant scenario can be identified.
Our distribution-free formulations currently exploit only the first-moment information of all possible distributions.
An intriguing direction for future research is to derive bounds incorporating higher-moment information and distributional assumptions.

%% file: appendix_tables.tex
\clearpage

\section{Deferred Proofs from Section~\ref{sec:solution_approach}}\label{sec:appendix_proofs}
\begin{proposition}\label{prop:optimality_cuts}
    Let $(y^*, \mu^*, w^*, u^*)$ be an optimal solution to the~\eqref{eq:rmp} and $\mu^* > \mathbb{E}_{\Xi \sim \mathbb{P}_{d,h}}[f(y^*, \Xi)]$ where $d$ and $h$ correspond to $w^*$ and $u^*$, respectively.
    Then, the optimality cut
    \begin{align*}
    \mu \le \ &\left[ 
        \sum_{i \in I} (1 - w_{i,d_i}) + \sum_{j \in J} \left(\sum  \nolimits_{z \in h(j)} ( 1-u_{j,z}) + \sum \nolimits_{z \in \mathcal{Z} \setminus h(j)} u_{j,z} \right) 
    \right] \cdot U \nonumber \\
    &+ \frac{1}{|S_{d,h}|} \sum_{\xi \in S_{d,h}} \left( \sum_{i \in I} \xi_{i}^{c} \; \alpha_{i}^{*, \xi} + \sum_{j \in J} \xi_{j}^{d} \; \beta_j^{*, \xi} + \sum_{(i,j) \in I \times J} M_{i,j}^\xi  \; \rho_{i,j}^{*, \xi} \; y_{i,j} \right),
\end{align*}
    is violated by $(y^*, \mu^*, w^*, u^*)$ and is valid for Problem~\eqref{eq:model}.
\end{proposition}
\begin{proof}{Proof.}
    By assumption, we have
    \begin{align*}
        \mu^* &> \mathbb{E}_{\Xi \sim \mathbb{P}_{d,h}}[f(y^*, \Xi)] = \frac{1}{|S_{d,h}|} \sum_{\xi \in S_{d,h}} f(y^*, \xi) \\
        &= \frac{1}{|S_{d,h}|} \sum_{\xi \in S_{d,h}} \left( \sum_{i \in I} \xi_{i}^{c} \; \alpha_{i}^{*, \xi} + \sum_{j \in J} \xi_{j}^{d} \; \beta_j^{*, \xi} + \sum_{(i,j) \in I \times J} M_{i,j}^\xi  \; \rho_{i,j}^{*, \xi} \; y^*_{i,j} \right),
    \end{align*}
    where the second equality follows from Problem~\eqref{eq:dual}.
    Hence, the optimality cut is violated by $(y^*, \mu^*, w^*, u^*)$.
    The validity of the cut follows from the validity of~\eqref{eq:final_benders_cuts} for Problem~\eqref{eq:model}.\qedhere
\end{proof}

\begin{proposition}\label{prop:convergence}
    Algorithm~\ref{alg:decomposition} converges to an optimal solution of Problem~\eqref{eq:model} in a finite number of iterations.
\end{proposition}
\begin{proof}{Proof.}
    For each distribution $\mathbb{P}_{d,h}$, the total number of optimality cuts that can be added to~\eqref{eq:rmp} is at most $|\vertices(\mathcal{T})|^{|S_{d,h}|}$.
    Furthermore, the number of possible distributions $\mathbb{P}_{d,h}$ is finite, hence Algorithm~\ref{alg:decomposition} terminates in finitely many iterations.
    The correctness follows from Proposition~\ref{prop:optimality_cuts} and the definition of $(\alpha^*, \beta^*, \rho^*)$ at Step~6 of Algorithm~\ref{alg:decomposition}.\qedhere
\end{proof}

\section{Exhaustive Distribution Enumeration Approach: Mathematical Model}\label{sec:appendix_model}

In each iteration of the EDE approach, the following standard two-stage stochastic program is solved for a fixed distribution $\mathbb{P}_{d,h}$:
\begin{alignat*}{10}
    \max_{y,\mu} \, & \ \mu - \sum_{i\in I}\sum_{j\in J} s_{i,j}\, y_{i,j} \nonumber \\
    \text{s.t.} \,  & \mu \le \frac{1}{|S_{d,h}|} \sum_{\xi \in S_{d,h}} \Big( \sum_{i \in I} \xi_{i}^{c} \ \alpha_{i}^{\xi} + \sum_{j \in J} \xi_{j}^{d} \ \beta_j^{\xi} + \sum_{(i,j) \in I \times J} M_{i,j} & & \rho_{i,j}^{\xi} \ y_{i,j} \Big) \\
    &&& \forall (\alpha, \beta, \rho) \in \vertices(\mathcal{T})^{S_{d,h}}, \\
    & \sum_{j \in J} y_{i,j} = d_{i}
    && \forall i \in I, \\
    & \sum_{i \in z} y_{i,j} \ge 1
    && \forall j \in J,\, z\in h(j), \\
    & y_{i,j} = 0
    && \forall i \in I, j\in J: i\notin\bigcup_{z\in h(j)} z,\\
    & y \in \{0,1\}^{I\times J}.
\end{alignat*}

\clearpage

\section{Tables~\ref{tab:both}, \ref{tab:stress_demand}, \ref{tab:dem}, and \ref{tab:sup}}\label{sec:appendix_tables}

\small{Not all combinations of node degrees and zone assignments correspond to $y$ ($|\mathcal{P}| \leq (|J|+1)^{|I|} \times 2^{2 |J|}$).
For solvable instances, the VD metric of Approach~(i) is the average $|\mathcal{P}|$ over 10 instances.}

\begin{table}[H]
\centering
\scriptsize
\caption{Performance summary of Approaches~(i)--(iv) under Regime~(1). For each instance size, 10 random instances are solved and average values are reported. TL indicates that the time limit was reached.
}
\label{tab:both}
\begin{tabular}{crlrrrrrr}
\toprule
Size & App. & Method & Gap & Itr. & Time & VD & RVSD & RVDD \\
\midrule
\multirow{4}{*}{(2,2)} & (i) & baseline & 0.00 & 17 & 0.29 & 16 & 1.00 & 1.00 \\
 & (ii) & \eqref{eq:jensen}, \eqref{eq:dominant_scenario} & 0.00 & 11 & 0.18 & 12 & 0.80 & 0.76 \\
 & (iii) & \eqref{eq:dff} & 0.00 & 7 & 0.12 & 8 & 0.56 & 0.50 \\
 & (iv) & \eqref{eq:dff}, \eqref{eq:dffs}, \eqref{eq:dffd} & 0.00 & 5 & 0.11 & 6 & 0.42 & 0.36 \\
\midrule
\multirow{4}{*}{(2,3)} & (i) & baseline & 0.00 & 65 & 1.07 & 64 & 1.00 & 1.00 \\
 & (ii) & \eqref{eq:jensen}, \eqref{eq:dominant_scenario} & 0.00 & 29 & 0.51 & 30 & 0.62 & 0.47 \\
 & (iii) & \eqref{eq:dff} & 0.00 & 26 & 0.49 & 27 & 0.57 & 0.42 \\
 & (iv) & \eqref{eq:dff}, \eqref{eq:dffs}, \eqref{eq:dffd} & 0.00 & 19 & 0.40 & 20 & 0.46 & 0.32 \\
\midrule
\multirow{4}{*}{(2,4)} & (i) & baseline & 0.00 & 257 & 4.71 & 256 & 1.00 & 1.00 \\
 & (ii) & \eqref{eq:jensen}, \eqref{eq:dominant_scenario} & 0.00 & 77 & 1.48 & 78 & 0.54 & 0.30 \\
 & (iii) & \eqref{eq:dff} & 0.00 & 68 & 1.35 & 69 & 0.48 & 0.27 \\
 & (iv) & \eqref{eq:dff}, \eqref{eq:dffs}, \eqref{eq:dffd} & 0.00 & 54 & 1.13 & 56 & 0.39 & 0.22 \\
\midrule
\multirow{4}{*}{(3,3)} & (i) & baseline & 0.00 & 352 & 6.75 & 304 & 1.00 & 1.00 \\
 & (ii) & \eqref{eq:jensen}, \eqref{eq:dominant_scenario} & 0.00 & 137 & 2.74 & 100 & 0.46 & 0.44 \\
 & (iii) & \eqref{eq:dff} & 0.00 & 114 & 2.28 & 82 & 0.29 & 0.31 \\
 & (iv) & \eqref{eq:dff}, \eqref{eq:dffs}, \eqref{eq:dffd} & 0.00 & 69 & 1.46 & 47 & 0.17 & 0.25 \\
\midrule
\multirow{4}{*}{(3,4)} & (i) & baseline & 0.00 & 2045 & 46.78 & 1664 & 1.00 & 1.00 \\
 & (ii) & \eqref{eq:jensen}, \eqref{eq:dominant_scenario} & 0.00 & 656 & 15.40 & 406 & 0.41 & 0.30 \\
 & (iii) & \eqref{eq:dff} & 0.00 & 633 & 14.82 & 387 & 0.33 & 0.25 \\
 & (iv) & \eqref{eq:dff}, \eqref{eq:dffs}, \eqref{eq:dffd} & 0.00 & 347 & 8.46 & 204 & 0.23 & 0.20 \\
\midrule
\multirow{4}{*}{(3,5)} & (i) & baseline & 0.00 & 10341 & 284.36 & 8704 & 1.00 & 1.00 \\
 & (ii) & \eqref{eq:jensen}, \eqref{eq:dominant_scenario} & 0.00 & 3407 & 92.96 & 2095 & 0.46 & 0.29 \\
 & (iii) & \eqref{eq:dff} & 0.00 & 3361 & 91.78 & 2053 & 0.41 & 0.27 \\
 & (iv) & \eqref{eq:dff}, \eqref{eq:dffs}, \eqref{eq:dffd} & 0.00 & 2216 & 61.77 & 1209 & 0.29 & 0.15 \\
\midrule
\multirow{4}{*}{(3,6)} & (i) & baseline & 2.86 & 46702 & TL & 40649 & 0.97 & 1.00 \\
 & (ii) & \eqref{eq:jensen}, \eqref{eq:dominant_scenario} & 0.00 & 12001 & 396.73 & 6139 & 0.32 & 0.17 \\
 & (iii) & \eqref{eq:dff} & 0.00 & 11942 & 398.07 & 6094 & 0.28 & 0.16 \\
 & (iv) & \eqref{eq:dff}, \eqref{eq:dffs}, \eqref{eq:dffd} & 0.00 & 7010 & 231.18 & 2975 & 0.18 & 0.10 \\
\midrule
\multirow{4}{*}{(4,4)} & (i) & baseline & 0.00 & 11586 & 341.08 & 9613 & 1.00 & 1.00 \\
 & (ii) & \eqref{eq:jensen}, \eqref{eq:dominant_scenario} & 0.00 & 2638 & 74.81 & 1437 & 0.17 & 0.32 \\
 & (iii) & \eqref{eq:dff} & 0.00 & 2552 & 71.93 & 1393 & 0.13 & 0.24 \\
 & (iv) & \eqref{eq:dff}, \eqref{eq:dffs}, \eqref{eq:dffd} & 0.00 & 959 & 27.60 & 490 & 0.06 & 0.19 \\
\midrule
\multirow{4}{*}{(4,5)} & (i) & baseline & 4.56 & 45042 & TL & 39935 & 0.78 & 1.00 \\
 & (ii) & \eqref{eq:jensen}, \eqref{eq:dominant_scenario} & 0.00 & 18260 & 769.40 & 10272 & 0.21 & 0.26 \\
 & (iii) & \eqref{eq:dff} & 0.00 & 18256 & 818.61 & 10228 & 0.19 & 0.24 \\
 & (iv) & \eqref{eq:dff}, \eqref{eq:dffs}, \eqref{eq:dffd} & 0.00 & 8006 & 324.11 & 4041 & 0.11 & 0.18 \\
\midrule
\multirow{4}{*}{(4,6)} & (i) & baseline & 4.80 & 42516 & TL & 41636 & 0.42 & 0.88 \\
 & (ii) & \eqref{eq:jensen}, \eqref{eq:dominant_scenario} & 0.17 & 37736 & TL & 20732 & 0.24 & 0.24 \\
 & (iii) & \eqref{eq:dff} & 0.16 & 37248 & TL & 19849 & 0.22 & 0.23 \\
 & (iv) & \eqref{eq:dff}, \eqref{eq:dffs}, \eqref{eq:dffd} & 0.10 & 32375 & 1693.17 & 16263 & 0.16 & 0.15 \\
\midrule
\multirow{4}{*}{(4,7)} & (i) & baseline & 4.92 & 39871 & TL & 39820 & 0.21 & 0.47 \\
 & (ii) & \eqref{eq:jensen}, \eqref{eq:dominant_scenario} & 0.22 & 36495 & TL & 25705 & 0.17 & 0.18 \\
 & (iii) & \eqref{eq:dff} & 0.22 & 36094 & TL & 26267 & 0.18 & 0.18 \\
 & (iv) & \eqref{eq:dff}, \eqref{eq:dffs}, \eqref{eq:dffd} & 0.16 & 32911 & TL & 20871 & 0.13 & 0.13 \\
\midrule
\multirow{4}{*}{(4,8)} & (i) & baseline & 5.46 & 36751 & TL & 36734 & 0.12 & 0.16 \\
 & (ii) & \eqref{eq:jensen}, \eqref{eq:dominant_scenario} & 0.26 & 32799 & TL & 27332 & 0.12 & 0.10 \\
 & (iii) & \eqref{eq:dff} & 0.27 & 33012 & TL & 28165 & 0.15 & 0.09 \\
 & (iv) & \eqref{eq:dff}, \eqref{eq:dffs}, \eqref{eq:dffd} & 0.21 & 30226 & TL & 24705 & 0.11 & 0.07 \\
\bottomrule
\end{tabular}
\end{table}

\begin{table}[h!]
\centering
\scriptsize
\caption{Performance summary of Approaches~(iii)--(vi) under Regime~(1).
For the 10 instances of size $(3,6)$, mean product demands increase from $100$ to $180$ units, while keeping demand CV and all other parameters unchanged.
Average values over 10 instances are reported.}
\label{tab:stress_demand}
\begin{tabular}{crlrrrrrr}
\toprule
Demand & App. & Formulation & Gap & Itr. & Time & VD & RVSD & RVDD \\
\midrule
\multirow{4}{*}{100}
 & (iii) & \eqref{eq:dff} & 0.00 & 11942 & 398.07 & 6094 & 0.28 & 0.16 \\
 & (iv) & \eqref{eq:dff}, \eqref{eq:dffs}, \eqref{eq:dffd} & 0.00 & 7010 & 231.18 & 2975 & 0.18 & 0.10 \\
 &(v) & \eqref{eq:dff}, \eqref{eq:dffd} & 0.00 & 7240 & 238.06 & 3140 &	0.19 & 0.10 \\
 &(vi) & \eqref{eq:dff}, \eqref{eq:dffs} & 0.00 & 11143	& 378.03 &	5566 & 0.25 & 0.15 \\
\midrule
\multirow{4}{*}{110}
 &(iii) & \eqref{eq:dff} & 0.00 & 12911 & 448.35 & 6722 & 0.30 & 0.18\\
 &(iv) & \eqref{eq:dff}, \eqref{eq:dffs}, \eqref{eq:dffd} & 0.00 & 9147 & 314.46 & 3813 & 0.21 & 0.12 \\
 &(v) & \eqref{eq:dff}, \eqref{eq:dffd} & 0.00 & 9309 & 318.50 & 3928 &	0.22 & 0.12\\
 &(vi) & \eqref{eq:dff}, \eqref{eq:dffs} & 0.00 & 12425 & 441.05 & 6355 & 0.28 & 0.18 \\
\midrule
\multirow{4}{*}{120}
 &(iii) & \eqref{eq:dff} & 0.00 & 15004 & 538.44 & 8091 & 0.33 & 0.24 \\
 &(iv) & \eqref{eq:dff}, \eqref{eq:dffs}, \eqref{eq:dffd} & 0.00 & 11061 & 391.66 &	4786 & 0.24 & 0.13 \\
 &(v) & \eqref{eq:dff}, \eqref{eq:dffd} & 0.00 & 11402 & 402.53 & 5078 & 0.27 &	0.13 \\
 &(vi) & \eqref{eq:dff}, \eqref{eq:dffs} & 0.00 & 13990 & 509.69 & 7222 & 0.28 & 0.22 \\
\midrule
\multirow{4}{*}{130}
 &(iii) & \eqref{eq:dff} & 0.00 & 17372	& 652.20 & 9079 & 0.28 & 0.29 \\
 &(iv) & \eqref{eq:dff}, \eqref{eq:dffs}, \eqref{eq:dffd} & 0.00 & 12123 & 441.15 & 5211 & 0.21 & 0.17 \\
 &(v) & \eqref{eq:dff}, \eqref{eq:dffd} & 0.00 & 12964 & 476.21 & 5930 & 0.26 &	0.18 \\
 &(vi) & \eqref{eq:dff}, \eqref{eq:dffs} & 0.00 & 15585 & 593.04 & 7741 & 0.22 & 0.27 \\
\midrule
\multirow{4}{*}{140}
 &(iii) & \eqref{eq:dff} & 0.00 & 18878 & 701.57 & 9876 & 0.22 & 0.35 \\
 &(iv) & \eqref{eq:dff}, \eqref{eq:dffs}, \eqref{eq:dffd} & 0.00 & 12653 & 463.34 & 5398 & 	0.16 & 0.24\\
 &(v) & \eqref{eq:dff}, \eqref{eq:dffd} & 0.00 & 13977 & 513.60 & 6510 & 0.21 & 0.25 \\
 &(vi) & \eqref{eq:dff}, \eqref{eq:dffs} & 0.00 & 16415	& 607.97 &	8134 & 0.16 & 0.32\\
\midrule
\multirow{4}{*}{150}
 &(iii) & \eqref{eq:dff} & 0.00 & 20151 & 794.65 & 10811 & 0.18 & 0.38\\
 &(iv) & \eqref{eq:dff}, \eqref{eq:dffs}, \eqref{eq:dffd} & 0.00 & 14069 & 525.01 &	6521 &	0.11 & 0.32 \\
 &(v) & \eqref{eq:dff}, \eqref{eq:dffd} & 0.00 & 16053 & 608.51 & 8094 & 0.16 &	0.34\\
 &(vi) & \eqref{eq:dff}, \eqref{eq:dffs} & 0.00 & 17122	& 643.85 & 8512 & 0.12 & 0.34 \\
\midrule
\multirow{4}{*}{160}
 &(iii) & \eqref{eq:dff} & 0.00 & 19557	& 736.57 & 10401 & 0.14 & 0.42 \\
 &(iv) & \eqref{eq:dff}, \eqref{eq:dffs}, \eqref{eq:dffd} & 0.00 & 13670 & 496.99 & 6520 &	0.09 & 0.33 \\
 &(v) & \eqref{eq:dff}, \eqref{eq:dffd} & 0.00 & 16947 & 635.16 & 9095 & 0.13 &	0.37 \\
 &(vi) & \eqref{eq:dff}, \eqref{eq:dffs} & 0.01 & 15331 & 564.84 & 7208 & 0.09 & 0.37 \\
\midrule
\multirow{4}{*}{170}
 &(iii) & \eqref{eq:dff} & 0.00 & 18403 & 681.88 & 9602 & 0.11 & 0.46\\
 &(iv) & \eqref{eq:dff}, \eqref{eq:dffs}, \eqref{eq:dffd} & 0.00 & 12757 & 455.52 & 5744 & 0.07 & 0.34 \\
 &(v) & \eqref{eq:dff}, \eqref{eq:dffd} & 0.00 & 16263 & 603.00 & 8322 & 0.10 & 0.39 \\
 &(vi) & \eqref{eq:dff}, \eqref{eq:dffs} & 0.00 & 14235 & 505.64 & 6585 & 0.07 & 0.38 \\
\midrule
\multirow{4}{*}{180}
 &(iii) & \eqref{eq:dff} & 0.00 & 17208	& 622.08 & 9143 & 0.10 & 0.51 \\
 &(iv) & \eqref{eq:dff}, \eqref{eq:dffs}, \eqref{eq:dffd} & 0.00 & 11327 & 397.99 & 5022 &	0.06 & 0.35 \\
 &(v) & \eqref{eq:dff}, \eqref{eq:dffd} & 0.00 & 15215 & 554.40 & 7816 & 0.09 & 0.42 \\
 &(vi) & \eqref{eq:dff}, \eqref{eq:dffs} & 0.00 & 12883	& 449.23 & 6086 & 0.06 & 0.44 \\
\bottomrule
\end{tabular}
\end{table}

\begin{table}[!h]
\scriptsize
\centering
\caption{Performance summary of Algorithm~\ref{alg:decomposition} for Approaches~(i) to~(iv) under Regime~(2). For each instance size, average values over 10 random instances are reported. TL indicates that the time limit was reached.}
\label{tab:dem}
\begin{tabular}{crlrrrr}
\toprule
Size & App. & Method & Gap & Itr. & Time & RVDD \\
\midrule
\multirow{4}{*}{(2,2)} & (i) & baseline & 0.00 & 20 & 0.34 & 1.00 \\
 & (ii) & \eqref{eq:jensen}, \eqref{eq:dominant_scenario} & 0.00 & 12 & 0.20 & 0.78 \\
 & (iii) & \eqref{eq:dff} & 0.00 & 7 & 0.13 & 0.50 \\
 & (iv) & \eqref{eq:dff}, \eqref{eq:dffd} & 0.00 & 5 & 0.10 & 0.36 \\
\midrule
\multirow{4}{*}{(2,3)} & (i) & baseline & 0.00 & 69 & 1.19 & 1.00 \\
 & (ii) & \eqref{eq:jensen}, \eqref{eq:dominant_scenario} & 0.00 & 30 & 0.54 & 0.46 \\
 & (iii) & \eqref{eq:dff} & 0.00 & 24 & 0.43 & 0.38 \\
 & (iv) & \eqref{eq:dff}, \eqref{eq:dffd} & 0.00 & 18 & 0.37 & 0.30 \\
\midrule
\multirow{4}{*}{(2,4)} & (i) & baseline & 0.00 & 261 & 4.85 & 1.00 \\
 & (ii) & \eqref{eq:jensen}, \eqref{eq:dominant_scenario} & 0.00 & 76 & 1.50 & 0.30 \\
 & (iii) & \eqref{eq:dff} & 0.00 & 67 & 1.36 & 0.27 \\
 & (iv) & \eqref{eq:dff}, \eqref{eq:dffd} & 0.00 & 54 & 1.13 & 0.22 \\
\midrule
\multirow{4}{*}{(3,3)} & (i) & baseline & 0.00 & 123 & 2.54 & 1.00 \\
 & (ii) & \eqref{eq:jensen}, \eqref{eq:dominant_scenario} & 0.00 & 69 & 1.43 & 0.43 \\
 & (iii) & \eqref{eq:dff} & 0.00 & 56 & 1.17 & 0.31 \\
 & (iv) & \eqref{eq:dff}, \eqref{eq:dffd} & 0.00 & 42 & 0.93 & 0.25 \\
\midrule
\multirow{4}{*}{(3,4)} & (i) & baseline & 0.00 & 537 & 12.38 & 1.00 \\
 & (ii) & \eqref{eq:jensen}, \eqref{eq:dominant_scenario} & 0.00 & 234 & 5.62 & 0.27 \\
 & (iii) & \eqref{eq:dff} & 0.00 & 212 & 5.03 & 0.24 \\
 & (iv) & \eqref{eq:dff}, \eqref{eq:dffd} & 0.00 & 157 & 3.89 & 0.18 \\
\midrule
\multirow{4}{*}{(3,5)} & (i) & baseline & 0.00 & 1788 & 47.38 & 1.00 \\
 & (ii) & \eqref{eq:jensen}, \eqref{eq:dominant_scenario} & 0.00 & 747 & 20.23 & 0.26 \\
 & (iii) & \eqref{eq:dff} & 0.00 & 711 & 19.38 & 0.25 \\
 & (iv) & \eqref{eq:dff}, \eqref{eq:dffd} & 0.00 & 568 & 15.65 & 0.14 \\
\midrule
\multirow{4}{*}{(3,6)} & (i) & baseline & 0.00 & 6980 & 213.98 & 1.00 \\
 & (ii) & \eqref{eq:jensen}, \eqref{eq:dominant_scenario} & 0.00 & 2242 & 70.18 & 0.16 \\
 & (iii) & \eqref{eq:dff} & 0.00 & 2173 & 68.00 & 0.15 \\
 & (iv) & \eqref{eq:dff}, \eqref{eq:dffd} & 0.00 & 1653 & 52.59 & 0.09 \\
\midrule
\multirow{4}{*}{(4,4)} & (i) & baseline & 0.00 & 830 & 22.91 & 1.00 \\
 & (ii) & \eqref{eq:jensen}, \eqref{eq:dominant_scenario} & 0.00 & 460 & 12.97 & 0.31 \\
 & (iii) & \eqref{eq:dff} & 0.00 & 420 & 11.90 & 0.24 \\
 & (iv) & \eqref{eq:dff}, \eqref{eq:dffd} & 0.00 & 313 & 9.04 & 0.19 \\
\midrule
\multirow{4}{*}{(4,5)} & (i) & baseline & 0.00 & 3130 & 101.77 & 1.00 \\
 & (ii) & \eqref{eq:jensen}, \eqref{eq:dominant_scenario} & 0.00 & 1525 & 51.17 & 0.25 \\
 & (iii) & \eqref{eq:dff} & 0.00 & 1443 & 48.78 & 0.23 \\
 & (iv) & \eqref{eq:dff}, \eqref{eq:dffd} & 0.00 & 1182 & 41.03 & 0.18 \\
\midrule
\multirow{4}{*}{(4,6)} & (i) & baseline & 0.00 & 11782 & 476.77 & 1.00 \\
 & (ii) & \eqref{eq:jensen}, \eqref{eq:dominant_scenario} & 0.00 & 6040 & 266.51 & 0.23 \\
 & (iii) & \eqref{eq:dff} & 0.00 & 5826 & 263.62 & 0.23 \\
 & (iv) & \eqref{eq:dff}, \eqref{eq:dffd} & 0.00 & 5336 & 240.26 & 0.14 \\
\midrule
\multirow{4}{*}{(4,7)} & (i) & baseline & 1.09 & 34645 & 1775.81 & 1.00 \\
 & (ii) & \eqref{eq:jensen}, \eqref{eq:dominant_scenario} & 0.00 & 20513 & 1325.78 & 0.22 \\
 & (iii) & \eqref{eq:dff} & 0.01 & 20237 & 1310.74 & 0.22 \\
 & (iv) & \eqref{eq:dff}, \eqref{eq:dffd} & 0.00 & 17908 & 1168.17 & 0.13 \\
\midrule
\multirow{4}{*}{(4,8)} & (i) & baseline & 4.99 & 37026 & TL & 0.45 \\
 & (ii) & \eqref{eq:jensen}, \eqref{eq:dominant_scenario} & 0.17 & 28140 & TL & 0.17 \\
 & (iii) & \eqref{eq:dff} & 0.16 & 28229 & TL & 0.17 \\
 & (iv) & \eqref{eq:dff}, \eqref{eq:dffd} & 0.07 & 27397 & TL & 0.13 \\
\bottomrule
\end{tabular}
\end{table}

\begin{table}
\scriptsize
\centering
\caption{Performance summary of Algorithm~\ref{alg:decomposition} for Approaches~(i) to~(iv) under Regime~(3). For each instance size, average values over 10 random instances are reported. TL indicates that the time limit was reached.}
\label{tab:sup}
\begin{tabular}{crlrrrr}
\toprule
Size & App. & Method & Gap & Itr. & Time & RVSD \\
\midrule
\multirow{4}{*}{(2,2)} & (i) & baseline & 0.00 & 12 & 0.22 & 1.00 \\
 & (ii) & \eqref{eq:jensen}, \eqref{eq:dominant_scenario} & 0.00 & 8 & 0.15 & 0.72 \\
 & (iii) & \eqref{eq:dff} & 0.00 & 6 & 0.11 & 0.51 \\
 & (iv) & \eqref{eq:dff}, \eqref{eq:dffs} & 0.00 & 5 & 0.08 & 0.33 \\
\midrule
\multirow{4}{*}{(2,3)} & (i) & baseline & 0.00 & 38 & 0.70 & 1.00 \\
 & (ii) & \eqref{eq:jensen}, \eqref{eq:dominant_scenario} & 0.00 & 18 & 0.36 & 0.47 \\
 & (iii) & \eqref{eq:dff} & 0.00 & 12 & 0.24 & 0.32 \\
 & (iv) & \eqref{eq:dff}, \eqref{eq:dffs} & 0.00 & 11 & 0.26 & 0.34 \\
\midrule
\multirow{4}{*}{(2,4)} & (i) & baseline & 0.00 & 90 & 1.83 & 1.00 \\
 & (ii) & \eqref{eq:jensen}, \eqref{eq:dominant_scenario} & 0.00 & 34 & 0.75 & 0.39 \\
 & (iii) & \eqref{eq:dff} & 0.00 & 23 & 0.50 & 0.24 \\
 & (iv) & \eqref{eq:dff}, \eqref{eq:dffs} & 0.00 & 21 & 0.49 & 0.24 \\
\midrule
\multirow{4}{*}{(3,3)} & (i) & baseline & 0.00 & 120 & 2.48 & 1.00 \\
 & (ii) & \eqref{eq:jensen}, \eqref{eq:dominant_scenario} & 0.00 & 31 & 0.70 & 0.27 \\
 & (iii) & \eqref{eq:dff} & 0.00 & 14 & 0.33 & 0.10 \\
 & (iv) & \eqref{eq:dff}, \eqref{eq:dffs} & 0.00 & 14 & 0.33 & 0.09 \\
\midrule
\multirow{4}{*}{(3,4)} & (i) & baseline & 0.00 & 423 & 10.09 & 1.00 \\
 & (ii) & \eqref{eq:jensen}, \eqref{eq:dominant_scenario} & 0.00 & 142 & 3.60 & 0.20 \\
 & (iii) & \eqref{eq:dff} & 0.00 & 115 & 2.92 & 0.13 \\
 & (iv) & \eqref{eq:dff}, \eqref{eq:dffs} & 0.00 & 109 & 2.79 & 0.11 \\
\midrule
\multirow{4}{*}{(3,5)} & (i) & baseline & 0.00 & 1210 & 32.86 & 1.00 \\
 & (ii) & \eqref{eq:jensen}, \eqref{eq:dominant_scenario} & 0.00 & 545 & 14.99 & 0.14 \\
 & (iii) & \eqref{eq:dff} & 0.00 & 518 & 14.23 & 0.08 \\
 & (iv) & \eqref{eq:dff}, \eqref{eq:dffs} & 0.00 & 500 & 13.86 & 0.08 \\
\midrule
\multirow{4}{*}{(3,6)} & (i) & baseline & 0.00 & 2575 & 81.31 & 1.00 \\
 & (ii) & \eqref{eq:jensen}, \eqref{eq:dominant_scenario} & 0.00 & 1330 & 42.10 & 0.09 \\
 & (iii) & \eqref{eq:dff} & 0.00 & 1281 & 40.62 & 0.05 \\
 & (iv) & \eqref{eq:dff}, \eqref{eq:dffs} & 0.00 & 1234 & 39.54 & 0.04 \\
\midrule
\multirow{4}{*}{(4,4)} & (i) & baseline & 0.00 & 1494 & 41.70 & 1.00 \\
 & (ii) & \eqref{eq:jensen}, \eqref{eq:dominant_scenario} & 0.00 & 159 & 4.68 & 0.06 \\
 & (iii) & \eqref{eq:dff} & 0.00 & 112 & 3.28 & 0.03 \\
 & (iv) & \eqref{eq:dff}, \eqref{eq:dffs} & 0.00 & 111 & 3.30 & 0.03 \\
\midrule
\multirow{4}{*}{(4,5)} & (i) & baseline & 0.00 & 4699 & 159.51 & 1.00 \\
 & (ii) & \eqref{eq:jensen}, \eqref{eq:dominant_scenario} & 0.00 & 1015 & 34.35 & 0.05 \\
 & (iii) & \eqref{eq:dff} & 0.00 & 915 & 31.12 & 0.03 \\
 & (iv) & \eqref{eq:dff}, \eqref{eq:dffs} & 0.00 & 874 & 31.32 & 0.03 \\
\midrule
\multirow{4}{*}{(4,6)} & (i) & baseline & 0.00 & 18776 & 1136.96 & 1.00 \\
 & (ii) & \eqref{eq:jensen}, \eqref{eq:dominant_scenario} & 0.00 & 9533 & 436.58 & 0.04 \\
 & (iii) & \eqref{eq:dff} & 0.00 & 9367 & 446.97 & 0.02 \\
 & (iv) & \eqref{eq:dff}, \eqref{eq:dffs} & 0.00 & 8783 & 442.97 & 0.02 \\
\midrule
\multirow{4}{*}{(4,7)} & (i) & baseline & 0.10 & 30886 & TL & 1.00 \\
 & (ii) & \eqref{eq:jensen}, \eqref{eq:dominant_scenario} & 0.03 & 29920 & 1628.97 & 0.03 \\
 & (iii) & \eqref{eq:dff} & 0.03 & 29262 & 1635.43 & 0.02 \\
 & (iv) & \eqref{eq:dff}, \eqref{eq:dffs} & 0.04 & 28754 & 1619.23 & 0.02 \\
\midrule
\multirow{4}{*}{(4,8)} & (i) & baseline & 0.16 & 27240 & TL & 1.00 \\
 & (ii) & \eqref{eq:jensen}, \eqref{eq:dominant_scenario} & 0.05 & 32352 & TL & 0.02 \\
 & (iii) & \eqref{eq:dff} & 0.05 & 31534 & TL & 0.01 \\
 & (iv) & \eqref{eq:dff}, \eqref{eq:dffs} & 0.05 & 30889 & TL & 0.01 \\
\bottomrule
\end{tabular}
\end{table}